\documentclass[leqno,12pt]{amsart} 
\usepackage[top=30truemm,bottom=30truemm,left=25truemm,right=25truemm]{geometry}
\usepackage{amssymb}
\usepackage{amsmath}
\usepackage{amsthm}
\usepackage{amscd}
\usepackage{mathrsfs}
\usepackage{graphicx}
\usepackage[dvips]{color}
\usepackage[all]{xy}
\usepackage{url}
  \makeatletter

\@addtoreset{equation}{section}
\@namedef{subjclassname@2020}{\textup{2020} Mathematics Subject Classification}
\makeatother
\theoremstyle{plain} 
\newtheorem{theorem}{\indent\bf Theorem}[section]
\newtheorem{lemma}[theorem]{\indent\bf Lemma}

\newtheorem{proposition}[theorem]{\indent\bf Proposition}

\newtheorem{conjecture}[theorem]{\indent\bf Conjecture}
\theoremstyle{definition} 
\newtheorem{definition}[theorem]{\indent\bf Definition}
\newtheorem{remark}[theorem]{\indent\bf Remark}
\newtheorem{example}[theorem]{\indent\bf Example}

\newtheorem{question}[theorem]{\indent\bf Question}

\newtheorem*{notation}{\indent\bf Notation}
\newcommand{\ddbar}{\partial \bar{\partial}}
\newcommand{\dbar}{\bar{\partial}}
\newcommand{\ai}{\sqrt{-1}}

\newcommand{\llangle}{\langle\!\langle}
\newcommand{\rrangle}{\rangle\!\rangle}

\begin{document}
\pagestyle{plain}
\thispagestyle{plain}

\title[Nakano positivity and vanishing theorems]
{Nakano positivity of singular Hermitian metrics and vanishing theorems of Demailly-Nadel-Nakano type}

\author[T. INAYAMA]{Takahiro INAYAMA}
\address{Department of Mathematics\\
	Faculty of Science and Technology\\
Tokyo University of Science\\
2641 Yamazaki, Noda\\
Chiba, 278-8510\\
Japan
}
\email{inayama\_takahiro@ma.noda.tus.ac.jp}
\email{inayama570@gmail.com}

\begin{abstract}
In this article, we propose a definition of Nakano semi-positivity of singular Hermitian metrics on holomorphic vector bundles. 
By using this positivity notion, we establish $L^2$-estimates for holomorphic vector bundles with Nakano positive singular Hermitian metrics. 
We show vanishing theorems, which generalize both Nakano type and Demailly-Nadel type vanishing theorems. 
As applications, we specifically construct globally Nakano semi-positive singular Hermitian metrics for several bundles, and prove vanishing theorems associated with them. 
\end{abstract}


\maketitle
\setcounter{tocdepth}{2}
\tableofcontents

\section{Introduction}
In algebraic and complex geometry, positivity notions for holomorphic vector bundles have played an important role.
Among them, a notion of positivity for singular Hermitian metrics has produced many significant results. 
%
On holomorphic line bundles, positivity of a singular Hermitian metric corresponds to plurisubharmonicity of the local weight. 
Hence, we can apply complex analytic methods to the research in the field of complex algebraic geometry. 
For holomorphic vector bundles, notions of singular Hermitian metrics were also introduced and investigated (cf. \cite{deC98}, \cite{BP08}). 

However, it turns out that we cannot always define the curvature currents with measure coefficients \cite{Rau15}. 
Hence, we need to define positivity notions without using curvature currents. 
We have such a characterization for Griffiths semi-positivity or semi-negativity (see Proposition \ref{prop:grif}). 
On the other hand, it was not known the way to define Nakano positivity of singular Hermitian metrics without using the expression of the curvature currents. 

Our main purpose in this article is to propose definitions of Nakano semi-positivity of singular Hermitian metrics on vector bundles (Definition \ref{def:nakanosemipositive} and \ref{def:localnakano}) 
and to establish a vanishing theorem (Theorem \ref{thm:vanishing}), which generalizes both the Nakano and the Demailly-Nadel vanishing theorems. 
These definitions are based on  $L^2$-theoretic characterizations of positivity, which were recently developed by the authors in \cite{DWZZ18}, \cite{DWZZ19}, \cite{HI19}, \cite{DNW19}, \cite{DNWZ20}. 


Throughout this paper, we let $X$ be an $n$-dimensional complex manifold, 
let $E\to X$ be a holomorphic vector bundle of finite rank $r>0$, 
and let $h$ be a singular Hermitian metric on $E$ (see Definition \ref{def:shm}). 

%
First, modifying the optimal $L^2$-estimate condition in \cite{DNWZ20}, we define the following positivity notions. 

\begin{definition}\label{def:nakanosemipositive}
Suppose that $h$ is a Griffiths semi-positive singular Hermitian metric. 
We say that $h$ is {\it globally Nakano semi-positive in the sense of singular Hermitian metrics} or simply {\it globally Nakano semi-positive} if 
for any Stein coordinate $(\Omega, \iota)$ around any point $x\in X$ (see Definition \ref{def:steincoordinate}) such that $E|_{\iota (\Omega)}$ is trivial on $\iota(\Omega)$, 
for any K\"ahler form $\omega_\Omega$ on $\Omega$, 
for any smooth strictly plurisubharmonic function $\psi$ on $\Omega$, for any positive integer $q$ such that $1\leq q \leq n$, and for any $\dbar$-closed $f\in L^2_{(n, q)}(\Omega, \iota^\star E ; \omega_\Omega, (\iota^\star h)e^{-\psi})$, 
there exists $u\in L^2_{(n, q-1)}(\Omega, \iota^\star E ; \omega_\Omega, (\iota^\star h) e^{-\psi})$ satisfying $\dbar u =f$ and 
$$
\int_\Omega |u|^2_{(\omega_\Omega, \iota^\star h)}e^{-\psi}dV_{\omega_\Omega} \leq \int_\Omega \langle B^{-1}_{\omega_\Omega, \psi}f, f \rangle_{(\omega_\Omega, \iota^\star h)}e^{-\psi} dV_{\omega_\Omega}, 
$$
where 
$B_{\omega_\Omega, \psi} = [\ai \ddbar \psi\otimes Id_E, \Lambda_{\omega_\Omega} ]$. 
Here we suppose that the right-hand side is finite (for detailed notation, see Notation in Section \ref{sec:prelimi}).

\end{definition}

\begin{definition}\label{def:localnakano}
Suppose that $h$ is a Griffiths semi-positive singular Hermitian metric. 
We say that $h$ is {\it locally Nakano semi-positive in the sense of singular Hermitian metrics} or 
simply {\it locally Nakano semi-positive} if for any point $x\in X$, there exists an open neighborhood $U$ of $x$ such that 
for any Stein coordinate $(\Omega, \iota)$ around $x$ such that $\iota(\Omega)\subset U$ and $E|_{\iota(\Omega)}$ is trivial, 
the condition in Definition \ref{def:nakanosemipositive} is satisfied on $\Omega$. 
\end{definition}

The condition in Definition \ref{def:nakanosemipositive} is a global property and the condition in 
Definition \ref{def:localnakano} is a local property. 
We clearly see that global Nakano semi-positivity implies local Nakano semi-positivity.
For smooth Hermitian metrics, the above definitions are equivalent (see Proposition \ref{prop:nakano}). 
We consider globally Nakano semi-positive singular Hermitian metrics in this article. 
We propose a problem related to the difference between the above definitions (see Question \ref{ques:globalocal}).


We explain the reason that we use the above condition to define Nakano positivity in Section \ref{sec:prelimi}. 
Here we only assume that $X$ is a complex manifold, not Hermitian or K\"ahler. 
Hence, we can define Nakano semi-positivity in a general setting. 
That is one of the advantages of Definition \ref{def:nakanosemipositive} and \ref{def:localnakano}.


In this setting, we can show the following result, which is a generalization of Demailly and Skoda's theorem \cite{DS} in the singular setting.

\begin{theorem}\label{thm:demaillyskoda}
Let $h$ be a Griffiths semi-positive singular Hermitian metric on $E$. 
Then $h\otimes \det h$ is globally Nakano semi-positive on $E\otimes \det E$. 
We can see that $h\otimes \det h$ is locally Nakano semi-positive as well.
\end{theorem}

Next, we consider the case that $X$ admits a K\"ahler metric $\omega_X$. 
In this situation, we can define strict Nakano positivity for singular Hermitian metrics in a simple way (see Definition \ref{def:strictnakano}). 
By using this notion, we prove the following $L^2$-estimate.
For simplicity, unless otherwise stated, if $X$ is a projective manifold, we fix the K\"ahler metric $\omega_X$ on $X$ as $\omega_X=\sqrt{-1}\Theta_{h_A}$ for some ample line bundle $(A, h_A)$ (since $\omega_X$ is only referenced to define strict Griffiths or Nakano $\delta_{\omega_X}$-positivity on $X$, this  does not alter the situation in any meaningful way). 

\begin{theorem}\label{thm:l2}
Let $(X, \omega_X)$ be a projective manifold and a K\"ahler metric on $X$, and $q$ be a positive integer. We assume that $(E, h)$ is globally strictly Nakano $\delta_{\omega_X}$-positive in the sense of Definition \ref{def:strictnakano} on $X$. 
Then for any $\dbar$-closed $f\in L^2_{(n, q)}(X, E ; \omega_X, h)$, there exists $u\in L^2_{(n, q-1)}(X, E ; \omega_X, h)$ satisfying $\dbar u=f$ and 
$$
\int_X |u|^2_{(\omega_X, h)}dV_{\omega_X} \leq \frac{1}{\delta q}\int_X |f|^2_{(\omega_X, h)}dV_{\omega_X}.
$$
\end{theorem}

We also get the following vanishing theorem,
which is a generalization of both the Nakano vanishing theorem and the Demailly-Nadel vanishing theorem. 
\begin{theorem}\label{thm:vanishing}
Let $(X, \omega_X)$ be a projective manifold and a K\"ahler metric on $X$. We assume that $(E, h)$ is globally strictly Nakano $\delta_{\omega_X}$-positive in the sense of Definition \ref{def:strictnakano} on $X$. 
Then the $q$-th cohomology group of $X$ with coefficients in the sheaf of germs of holomorphic sections of $K_X\otimes \mathscr{E}(h)$ vanishes for $q>0 :$ 
$$
H^q(X, K_X\otimes \mathscr{E}(h))=0, 
$$
where $\mathscr{E}(h)$ is the sheaf of germs of locally square integrable holomorphic sections of $E$ with respect to $h$. 
\end{theorem}

Here, we can prove that the sheaf $\mathscr{E}(h)$ is coherent when $h$ is a Nakano (semi-)positive singular Hermitian metric (see Proposition \ref{prop:coherence}).
As an application of Theorem \ref{thm:demaillyskoda} and Theorem \ref{thm:vanishing}, we get the following result. 

\begin{theorem}\label{cor:grif}
Let $(X, \omega_X)$ be a projective manifold and a K\"ahler metric on $X$. We assume that $h$ is strictly Griffiths $\delta_{\omega_X}$-positive on $X$ (see Definition \ref{def:strictgrif}). 
Then the $q$-th cohomology group of $X$ with coefficients in the sheaf of germs of holomorphic sections of $K_X\otimes \mathscr{E}(h\otimes \det h)$ vanishes for $q>0 :$ 
$$
H^q(X, K_X\otimes \mathscr{E}(h\otimes \det h))=0.
$$ 
\end{theorem}

Theorem \ref{cor:grif} can be regarded as a generalization of the Griffiths vanishing theorem (cf. \cite[Chapter VII, Corollary 9.4]{DemCom}).
If the Lelong number $\nu (\det h, x)<1$ for all points $x\in X$, this kind of result was obtained in \cite[Corollary 1.4]{Ina18}. 
We stress that, although Definition \ref{def:nakanosemipositive} or \ref{def:localnakano} is one choice of the definition of singular Nakano semi-positivity, Theorem \ref{cor:grif} is independent of these choices. 
Our formulation fits with the classical framework in this sense. 
Indeed, we can prove that the global Nakano semi-positivity, the local Nakano semi-positivity and Griffiths semi-positivity  for singular Hermitian metrics are all identical when $\dim X=1$ or rank~$E=1$ (see Section \ref{sec:property}).
Using Theorem \ref{thm:vanishing} and \ref{cor:grif}, we can determine the non-existence of Nakano or Griffiths positive singular Hermitian metrics on certain vector bundles (see Example \ref{examp:rei}).

As applications, we have the following results. 
First, we show the singular Nakano semi-positivity of the following direct image bundle. 
\begin{theorem}\label{mainthm:singber}
		Let $U\subset \mathbb{C}^n_{\{ t\}}$ and $\Omega \subset \mathbb{C}^m_{\{ z\}}$ be bounded domains, and $\varphi$ be a locally bounded plurisubharmonic function on $\overline{U\times \Omega}$.
	We also let $\Omega$ be pseudoconvex. 
	Set, for each $t\in U$, $A^2_t:= \{ f\in \mathscr{O}(\Omega) \mid \| f\|_t^2:=\int_\Omega |f|^2e^{-\varphi(t, \cdot)} <+\infty \}$ and $A^2:=\coprod_{t\in U}A^2_t $.  
	Then the trivial vector bundle $(A^2, \| \cdot \|)$ is globally Nakano semi-positive in the sense of Definition \ref{def:nakanosemipositive}. 
\end{theorem}

This theorem is well-known in the situation that $\varphi$ is smooth, which was obtained by Berndtsson \cite{Ber09}. 
A key ingredient to prove the theorem is that singular Nakano semi-positivity is preserved with respect to a increasing sequence (Proposition \ref{prop:increasing}).

Next, 
we show the following theorem. 

\begin{theorem}\label{thm:bigvanishing}
	Let $(X, \omega_X)$ be a projective manifold and a K\"ahler metric on $X$.
	We assume that $E\to X$ is a V-big vector bundle (see Definition \ref{def:lbigvbig}). 
	Then for any $m\in \mathbb{N}$, there exists a positive constant $\delta$ such that $S^mE\otimes \det E$ admits a globally strictly Nakano $\delta_{\omega_X}$-positive singular Hermitian metric $h_m$. 
	Here $S^m E$ is the $m$-th symmetric power of $E$. 
	Then we also obtain the following vanishing theorem 
	$$
	H^q(X, K_X\otimes S^m\mathscr{E}\otimes \det \mathscr{E}(h_m))=0
	$$
	for $m, q>0$, where $S^m\mathscr{E}\otimes \det \mathscr{E}(h_m)$ is the sheaf of germs of locally square integrable holomorphic sections of $S^mE\otimes \det E$ with respect to $h_m$. 
\end{theorem}

This is one application of our vanishing theorem.
This result was published by Iwai as \cite[Corollary 5.9]{Iwa18} (it was communicated to Iwai by the author). 

The organization of this paper is as follows. 
We start with Section \ref{sec:prelimi} a general discussion of smooth and singular Hermitian metrics on holomorphic vector bundles. 
Here we introduce several H\"ormander type conditions. 
In Section \ref{sec:dsk}, we explain the result of Demailly and Skoda. 
Here we also generalize the result in the singular setting. 
In Section \ref{sec:vanishing}, we establish $L^2$-estimates and vanishing theorems for holomorphic vector bundles with Nakano positive singular Hermitian metrics. 
In Section \ref{sec:property}, we verify that our definition of Nakano semi-positivity is an appropriate positivity notion when we compare it with the definition of Griffiths semi-positivity. 
In Section \ref{sec:applications}, we show applications of our main theorems and prove Theorem \ref{mainthm:singber} and \ref{thm:bigvanishing}. 
Finally, in Section \ref{sec:problems}, we propose some questions which might be worth thinking about. 

\vskip10mm
{\bf Acknowledgment. }
The author would like to thank his supervisor Prof. Shigeharu Takayama for enormous supports. 
He is also grateful to Dr. Genki Hosono for helpful comments. 
Last but not least, he is grateful to the anonymous referees for their  constructive criticism and valuable suggestions.
This work is supported by the Program for Leading Graduate Schools, MEXT, Japan. 
This work is also supported by JSPS KAKENHI Grant Number 18J22119. 

\section{Notation and Preliminaries}\label{sec:prelimi}
Throughout this paper, we use the following notation and definitions. 

\begin{notation}
\begin{itemize}
\item $K_X$ : the canonical line bundle of $X$.
\item $dV_\omega:=\frac{\omega^n}{n!}$ : the volume form determined by $\omega$. 
\item $E^\star$ : the dual bundle of $E$.
\item $h^\star$ : the dual metric of $h$ on $E^\star$.
\item $\mathscr{O}(E)$ : the sheaf of germs of local holomorphic sections of $E$.
\item $C^k_{(p, q)}(X, E):=C^k(X, \wedge^{(p, q)}T^\star_X\otimes E)$ for $0\leq k\leq +\infty$. 
\item $\mathscr{D}_{(p,q)}(X, E)$ : the space of smooth sections of $\wedge^{(p, q)}T^\star_X\otimes E$ with compact support. 
\item $L^p_{(p, q)}(X, E ; \omega, h)$ : the space of $L^p$ sections of $\wedge^{(p, q)}T^\star_X\otimes E$ with respect to $\omega$ and $h$. 
\item $\llangle \alpha, \beta \rrangle_{(\omega, h)} := \int_X \langle\alpha, \beta\rangle_{(\omega, h)}dV_\omega$.
\item $\| \alpha\|^2_{(\omega, h)}:= \llangle \alpha, \alpha \rrangle_{(\omega, h)}$.
\item $D'^\star_\psi$ : the adjoint operator of $D'_{\psi}$ with respect to $\llangle \cdot,\cdot  \rrangle_{(\omega, he^{-\psi})}$.
\item $\dbar^\star_\psi$ : the adjoint operator of $\dbar$ with respect to $\llangle \cdot,\cdot  \rrangle_{(\omega, he^{-\psi})}$.
\item $\Delta'_\psi:=D'_\psi D'^\star_\psi+D'^\star_\psi D'_\psi, \Delta_\psi''=\dbar \dbar^\star_\psi + \dbar^\star_\psi \dbar$ with respect to $\llangle \cdot,\cdot  \rrangle_{(\omega, he^{-\psi})}$. 
\item $L_\omega : C^\infty_{(p, q)}(X, E)\to C^\infty_{(p+1, q+1)}(X, E)$ : the operator defined by $\omega \wedge \cdot $.
\item $\Lambda_\omega$ : the adjoint operator of $L_\omega$. 
\item $[\cdot, \cdot]$ : the graded Lie bracket. 
\item $\Delta^n(p ; r):= \{ (z_1, \cdots z_n)\in \mathbb{C}^n \mid |z_i - p_i|< r \}$ for $p= (p_1, \cdots , p_n)\in \mathbb{C}^n$. 
\item $\Delta^n_r:= \Delta^n(0; r)$. 
\end{itemize}
\end{notation}

\begin{definition}\label{def:steincoordinate}
Let $\Omega$ be an $n$-dimensional Stein manifold and $\iota : \Omega \to X$ be a holomorphic map from $\Omega$ to $X$. 
We say that $(\Omega, \iota)$ is a {\it Stein coordinate} around $x_0\in X$ if and only if the following conditions are satisfied: 
\begin{enumerate}
\item $\iota:\Omega \to X$ is an injective holomorphic map, i.e. $\Omega \to \iota(\Omega)$ defines a biholomorphic map.
\item $\iota(\Omega)$ is an open subset of $X$ such that $x_0 \in \iota(\Omega)$.
\end{enumerate}
\end{definition}

By definition, every complex manifold admits a Stein coordinate around any point. 

\subsection{Smooth Hermitian metrics}
We explain some definitions and properties of smooth Hermitian metrics. 
In this subsection, we always assume that a Hermitian metric $h$ is smooth. 

Let $\Theta_{(E, h)}$ be the Chern curvature tensor of $(E, h)$. 
Taking a local coordinate $(z_1, \cdots , z_n)$ of $X$ and an orthonormal frame $(e_1, \cdots, e_r)$ of $E$, 
we can write 
$$
\sqrt{-1}\Theta_{(E, h)}= \sum_{1\leq j, k \leq n, 1\leq \lambda, \mu\leq r}c_{j\bar{k}\lambda\bar{\mu}}dz_j\wedge d\bar{z}_k \otimes e^\star_\lambda \otimes e_\mu.
$$
We identify the curvature tensor with a Hermitian form 
$$
\widetilde{\Theta}_{(E, h)}(\tau, \tau)= \sum_{1\leq j, k \leq n, 1\leq \lambda, \mu\leq r}c_{j\bar{k}\lambda\bar{\mu}}\tau_{j\lambda}\bar{\tau}_{k\mu}
$$
for $\tau = \sum_{j, \lambda}\tau_{j\lambda}\frac{\partial}{\partial z_i}\otimes e_\lambda\in T_X\otimes E$ on $T_X\otimes E$ . 
By using this Hermitian form, we define the following positivity notions. 

\begin{definition}
The Hermitian vector bundle $(E, h)$ is said to be : 
\begin{enumerate}
\item {\it Griffiths positive} (resp. {\it Griffiths negative}) if we have $\widetilde{\Theta}_{(E, h)}(\xi\otimes s, \xi\otimes s)>0$ (resp. $\widetilde{\Theta}_{(E, h)}(\xi\otimes s, \xi\otimes s)<0$)
for all non-zero elements $\xi\in T_X, s\in E$. We denote it by $\Theta_{(E, h)}>_{{\rm Grif.}}0$ (resp. $\Theta_{(E, h)}<_{{\rm Grif.}}0$).
\item {\it Nakano positive} (resp. {\it Nakano negative}) if we have $\widetilde{\Theta}_{(E, h)}(\tau, \tau)>0$ (resp. $\widetilde{\Theta}_{(E, h)}(\tau, \tau)<0$) for all non-zero elements $\tau \in T_X\otimes E$. 
We denote it by $\Theta_{(E, h)}>_{{\rm Nak.}}0$ (resp. $\Theta_{(E, h)}<_{{\rm Nak.}}0$).
\end{enumerate}
Corresponding semi-positivity and semi-negativity are defined by relaxing the strict inequalities. 
\end{definition}

We can associate the Nakano positivity of $(E,h)$ with the positivity of the operator $[\ai \Theta_{(E, h)}, \Lambda_\omega]$ from the following lemma. 
\begin{lemma}$($cf. \cite[Chapter VII, Lemma 7.2]{DemCom}, \cite[Lemma 2.5]{DNWZ20}$)$\label{lem:operator}
Let $(X, \omega)$ be a K\"ahler manifold. 
We have that $(E, h)>_{{\rm Nak.}}0$ $($resp. $(E, h)\geq_{{\rm Nak.}}0$$)$ if and only if the Hermitian operator $[\ai \Theta_{(E, h)}, \Lambda_\omega]$ is positive 
definite $($resp. semi-positive definite$)$ on $\wedge^{(n,1)}T^\star_X\otimes E$. 
\end{lemma}

We can define Griffiths positivity and negativity without using the curvature tensor. 
We have the following result. 
\begin{proposition}$($cf. \cite[Section 2]{Rau15}$)$\label{prop:grif}
The following properties are equivalent{\rm :} 
\begin{enumerate}
\item h is Griffiths semi-negative.
\item $|u|^2_h$ is plurisubharmonic for any local holomorphic section $u$ of $E$.
\item $\log |u|^2_h$ is plurisubharmonic for any local holomorphic section $u$ of $E$. 
\item the dual metric $h^\star$ on $E^\star$ is Griffiths semi-positive. 
\end{enumerate}
\end{proposition}

We can treat the above conditions $(2)$ and $(3)$ without using the curvature tensor. 
Hence, we use these conditions to define Griffiths semi-positivity and semi-negativity of singular Hermitian metrics (see Definition \ref{GrifPos}). 
On the other hand, we did not know such a characterization of Nakano positivity. 

Recently, new positivity notions defined via the H\"ormander $L^p$-estimate were widely investigated. 
These studies can be regarded as a converse of H\"ormander's estimate which is essentially due to Andreotti and Vesentini \cite{AV65}, and H\"ormander \cite{Hor65} (see also Theorem \ref{thm:demailly}).
Initially, Berndtsson established a converse of H\"ormander's $L^2$-estimate for a continuous function on a $1$-dimensional domain, and use this result to prove the complex Pr\'ekopa theorem in \cite{Ber98}. 
In \cite{HI19}, we introduced the twisted H\"ormander condition for holomorphic vector bundles on an $n$-dimensional domain. 

\begin{definition}$($\cite[Definition 3.3]{HI19}$)$\label{def:twistedhorm}
Let $h$ be a singular Hermitian metric on $E\to \Omega$ over a domain $\Omega \subset \mathbb{C}^n$. 
We say that $(E, h)$ satisfies {\it the twisted H\"ormander condition} if for any positive integer $m$, for any smooth strictly plurisubharmonic function $\psi$ on $\Omega$, 
and for any $\dbar$-closed $f=\sum_j f_j dz_1 \wedge \cdots \wedge dz_n\wedge d\bar{z}_j\in \mathscr{D}_{(n, 1)}(\Omega, E^{\otimes m})$, there exists $u\in C^\infty_{(n, 0)}(\Omega, E^{\otimes m})$ satisfying $\dbar u=f$ and 
$$
\int_\Omega |u|^2_{(\omega_\Omega, h^{\otimes m})}e^{-\psi}dV_{\omega_\Omega}\leq \int_\Omega \sum_{1\leq i, j\leq n}\langle \psi^{i\bar{j}}f_i, f_j  \rangle_{(\omega_\Omega, h^{\otimes m})}e^{-\psi}dV_{\omega_\Omega}, 
$$
where we assume that the right-hand side is finite. 
Here $(\psi^{i\bar{j}})_{1\leq i, j \leq n}$ denotes the inverse matrix of $(\frac{\partial^2}{\partial z_i \partial \bar{z}_j})_{1\leq i, j \leq n}$. 
\end{definition}

We remark that the matrix $(\psi^{i\bar{j}})_{1\leq i, j \leq n}$ corresponds to the inverse operator of $B_{\omega_\Omega, \psi}=[\ai \ddbar \psi \otimes Id_{E^{\otimes m}}, \Lambda_{\omega_\Omega}]$. 
We proved that this twisted H\"ormander condition implies Griffiths semi-positivity under some regularity assumptions (\cite[Theorem 3.5]{HI19}, see also \cite[Theorem 1.2]{DNWZ20}). 

Then Deng, Ning, Wang, and Zhou introduced and improved various H\"ormander type positivity notions for holomorphic vector bundles, which were called the multiple coarse $L^p$-estimate condition and the optimal $L^p$-estimate condition in \cite{DNWZ20}.
We mention that the twisted H\"ormander condition above is something like a multiple optimal $L^2$-estimate type condition. 
In this paper, we focus on the optimal $L^2$-estimate condition.

\begin{definition}$($\cite[Definition 1.1]{DNWZ20}$)$\label{def:optimaltheta}
Assume that a K\"ahler manifold $(X, \omega)$ admits a positive holomorphic line bundle, $(E, h)$ is a (singular) Hermitian vector bundle (maybe of infinite rank) over X. 
Then we say that $(E, h)$ satisfies {\it the optimal $L^2$-estimate condition} 
if for any positive holomorphic line bundle $(A, h_A)$ on $X$, for any $\dbar$-closed $f\in \mathscr{D}_{(n, 1)}(X, E\otimes A)$, 
there exists $u\in L^2_{(n, 0)}(X, E\otimes A)$ satisfying $\dbar u =f$ and 
$$
\int_X |u|^2_{(\omega, h\otimes h_A)}dV_\omega \leq \int_X \langle B_{h_A}^{-1}f, f\rangle_{(\omega, h\otimes h_A)}dV_\omega, 
$$
where $B_{h_A}=[\sqrt{-1}\Theta_{(A, h_A)}\otimes Id_E, \Lambda_\omega]$ and we assume that the right-hand side is finite. 
\end{definition}

Furthermore, they succeeded in characterizing Nakano semi-positivity by using the above condition. 
To be precise, they proved the following theorem.

\begin{theorem}$($\cite[Theorem 1.1]{DNWZ20}$)$\label{thm:nakanotheta}
Suppose that a K\"ahler manifold $(X, \omega)$ admits a positive holomorphic line bundle, $(E, h)$ is a smooth Hermitian vector bundle over $X$, 
and $\theta \in C^0_{(1, 1)}(X, End(E))$ with $\theta^\star=\theta$. 
We assume that for any $\dbar$-closed $f\in \mathscr{D}_{(n, 1)}(X, E\otimes A)$, and for any positive holomorphic line bundle $(A, h_A)$ such that $\sqrt{-1}\Theta_{(A, h_A)}\otimes Id_E + \theta >_{{\rm Nak.}}0$ on $supp f$, 
there exists $u\in L^2_{(n, 0)}(X, E\otimes A)$ satisfying $\dbar u =f $ and 
$$
\int_X |u|^2_{(\omega, h\otimes h_A)}dV_\omega \leq \int_X \langle B^{-1}_{h_A, \theta}f, f\rangle_{(\omega, h\otimes h_A)}dV_\omega, 
$$
where $B_{h_A, \theta}=[\sqrt{-1}\Theta_{(h_A, \theta)}\otimes Id_E + \theta, \Lambda_\omega]$ and we assume that the right-hand side is finite. 
Then $\sqrt{-1}\Theta_{(E, h)}\geq_{{\rm Nak.}}\theta$.
\end{theorem}

Here we consider the case that $\theta=0$. In this situation, the condition in Theorem \ref{thm:nakanotheta} is just the optimal $L^2$-estimate condition introduced in Definition \ref{def:optimaltheta}. 
By applying and modifying this theorem, we get the following proposition. 

\begin{proposition}\label{prop:nakano}
Let 
$h$ be a smooth Hermitian metric on $E$. 
We consider the following conditions{\rm :}  
\begin{enumerate}
\item $h$ is Nakano semi-positive. 
\item For any Stein coordinate $(\Omega, \iota)$ such that $E|_{\iota (\Omega)}$ is trivial on $\iota( \Omega)$, 
for any K\"ahler form $\omega_\Omega$ on $\Omega$, 
for any smooth strictly plurisubharmonic function $\psi$ on $\Omega$, for any positive integer $q$ such that $1\leq q\leq n$, 
and for any $\dbar$-closed $f\in L^2_{(n, q)}(\Omega, \iota^\star E ; \omega_\Omega, (\iota^\star h)e^{-\psi})$, 
there exists $u\in  L^2_{(n, q-1)}(\Omega, \iota^\star E ; \omega_\Omega, (\iota^\star h)e^{-\psi})$ satisfying $\dbar u =f$ and 
$$
\int_\Omega |u|^2_{(\omega_\Omega, \iota^\star h)}e^{-\psi}dV_{\omega_\Omega} \leq \int_\Omega \langle B^{-1}_{\omega_\Omega, \psi}f, f \rangle_{(\omega_\Omega, \iota^\star h)}e^{-\psi} dV_{\omega_\Omega}, 
$$
provided the right-hand side is finite. 
\item $(E, h)$ satisfies the optimal $L^2$-estimate condition. 
\end{enumerate}
Then the condition $(1)$ is equivalent to the condition $(2)$. If $X$ admits a K\"ahler metric $\omega$ and a positive holomorphic line bundle on $X$, 
the above three conditions are equivalent. 
\end{proposition}
Obviously, the above condition $(2)$ corresponds to the condition in Definition \ref{def:nakanosemipositive}. 
Theorem \ref{thm:nakanotheta} and the following Theorem \ref{thm:demailly} imply that the condition $(1)$ is equivalent to the condition $(3)$. 
The way to prove that the condition $(1)$ is equivalent to the condition $(2)$ is essentially contained in the proof of Theorem \ref{thm:nakanotheta} in \cite{DNWZ20}. 
However, our situation is slightly different. 
Hence, for the sake of completeness, we show the equivalence of $(1)$ and $(2)$ here. 
In our situation, the proof is a little bit simpler. 
Before proving that, we prepare the following $L^2$-estimate theorem.

\begin{theorem}$($cf. \cite{Dem82}, \cite[Chapter VIII, Theorem 6.1]{DemCom}$)$\label{thm:demailly}
Let $(X, \widehat{\omega})$ be a complete K\"ahler manifold, $\omega$ be another K\"ahler metric which is not necessarily complete, and $(E, h)\to X$ be Nakano semi-positive vector bundle. 
We also let $A_{q,\omega, h}=[\sqrt{-1}\Theta_{(E, h)}, \Lambda_\omega]$ be the operator in bidegree $(n, q)$ for $q\geq 1$. 
Then for any $\dbar$-closed $f\in L^2_{(n, q)}(X, E ; \omega, h)$, there exists $u\in L^2_{(n, q-1)}(X, E ; \omega, h)$
satisfying $\dbar u=f$ and 
$$
\int_X |u|^2_{(\omega, h)}dV_\omega \leq \int_X \langle A_{q,\omega,  h}^{-1}f, f\rangle_{(\omega, h)}dV_\omega, 
$$
where we assume that the right-hand side is finite. 
\end{theorem}

\begin{proof}[\indent\sc Proof of Proposition \ref{prop:nakano}]

First, we assume that $h$ is Nakano semi-positive. We take an arbitrary Stein coordinate $(\Omega, \iota )$ such that $E|_{\iota (\Omega)}$ is trivial on $\iota(\Omega)$,
an arbitrary K\"ahler metric $\omega_\Omega$ on $\Omega$, 
and an arbitrary smooth strictly plurisubharmonic function $\psi $ on $\Omega$. 
Considering the twisted weight $(\iota^\star h)e^{-\psi}$, 
we have that $\sqrt{-1}\Theta_{(\iota^\star E, (\iota^\star h)e^{-\psi})}=\sqrt{-1}\Theta_{(\iota^\star E, (\iota^\star h))}+ \ai \ddbar\psi \otimes Id_{\iota^\star E} $ and 
\begin{align*}
A_{q, \omega_\Omega, (\iota^\star h)e^{-\psi}}&=[\ai \Theta_{(\iota^\star E, \iota^\star h)}, \Lambda_{\omega_\Omega}]+ [\ai \ddbar \psi \otimes Id_{\iota^\star E}, \Lambda_{\omega_\Omega}]\\
&=A_{q,\omega_\Omega, {\iota^\star h}}+B_{\omega_\Omega, \psi}.
\end{align*}
We have $(\iota^\star h)e^{-\psi}$ is Nakano positive on $\iota^\star E$. 
Then Theorem \ref{thm:demailly} implies that for any $q\geq 1$ and for any $\dbar$-closed $f\in L^2_{(n, q)}(\Omega, \iota^\star E ; \omega_\Omega, (\iota^\star h)e^{-\psi})$, 
we have $u\in L^2_{(n, q-1)}(\Omega, \iota^\star E ; \omega_\Omega, (\iota^\star h)e^{-\psi})$ satisfying $\dbar u =f$ and 
$$
\int_\Omega |u|^2_{(\omega_\Omega, \iota^\star h)}e^{-\psi} dV_{\omega_\Omega} \leq \int_\Omega \langle A^{-1}_{q,\omega_\Omega, (\iota^\star h)e^{-\psi}}f, f \rangle_{(\omega_\Omega, \iota^\star h)}e^{-\psi}dV_{\omega_\Omega}.
$$
Since $\iota^\star h$ is also Nakano semi-positive, we have the inequality 
$$
\langle A^{-1}_{q,\omega_\Omega, (\iota^\star h)e^{-\psi}}f, f\rangle_{(\omega_\Omega, \iota^\star h)}  \leq \langle B^{-1}_{\omega_\Omega, \psi}f, f\rangle_{(\omega_\Omega, \iota^\star h)}. 
$$
Therefore, we also have the estimate 
$$
\int_\Omega |u|^2_{(\omega_\Omega, \iota^\star h)}e^{-\psi} dV_{\omega_\Omega}\leq \int_\Omega \langle B^{-1}_{\omega_\Omega, \psi}f, f \rangle_{(\omega_\Omega, \iota^\star h)}e^{-\psi}dV_{\omega_\Omega}.
$$

Next, we assume that the condition $(2)$. Suppose that $h$ is not Nakano semi-positive at some point $x_0\in X$.
We take a Stein coordinate $(\Delta^n_r, \iota )$ such that $\iota(0)=x_0$ and $E|_{\iota(\Delta^n_r)}$ is trivial for some $r>0$, take the standard K\"ahler metric $\omega_0=\ai \ddbar |z|^2$ on $\Delta^n_r$, and take a frame $(e_1, \cdots, e_r)$ of $\iota^\star E$ on $\Delta^n_r$ such that $(e_1, \cdots , e_r)$ is orthonormal at $0\in \Delta^n_r$.
Then $(\iota^\star E, \iota^\star h)$ is not Nakano semi-positive at $0\in \Delta^n_r $.
For the sake of simplicity, we also write $(E, h)(=(\iota^\star E, \iota^\star h))$ on $\Delta^n_r$. 
Note that, by Lemma \ref{lem:operator}, the operator $[\ai\Theta_{(E,h)}, \Lambda_{\omega_0}]$ is not semi-positive definite at $0\in \Delta^n_r$.
Then there exists $f_0\in \wedge^{(n,1)}T^\star_{\Delta^n_r, 0}\otimes E_{0}$ such that 
$$
\langle [\ai\Theta_{(E,h)}, \Lambda_{\omega_0}]f_0, f_0 \rangle_{(\omega_0, h)}=\langle A_{1, \omega_0, h}f_0, f_0\rangle_{(\omega_0, h)}<0. 
$$

We fix a smooth strictly plurisubharmonic function $\psi$ on $\Delta^n_r$. 
Then for any $\dbar$-closed $f\in \mathscr{D}_{(n, 1)}(\Delta^n_r, E)\subset L^2_{(n, 1)}(\Delta^n_r, E ; \omega_0, he^{-\psi})$, 
there exists $u\in C^\infty_{(n, 0)}(\Delta^n_r, E) $ satisfying $\dbar u =f$ and 
$$
\int_{\Delta^n_r}|u|^2_{(\omega_0, h)}e^{-\psi}dV_{\omega_0}\leq \int_{\Delta^n_r} \langle B^{-1}_{\omega_0, \psi}f, f \rangle_{(\omega_0, h)}e^{-\psi}dV_{\omega_0}. 
$$
Therefore, we have 
\begin{align*}
|\llangle B^{-1}_{\omega_0, \psi}f, f \rrangle_{(\omega_0, he^{-\psi})} |^2 &= |\llangle B^{-1}_{\omega_0, \psi}f, \dbar u \rrangle_{(\omega_0, he^{-\psi})} |^2 \\
& =|\llangle \dbar^\star_{\psi}(B^{-1}_{\omega_0, \psi}f), u \rrangle_{(\omega_0, he^{-\psi})} |^2 \\
& \leq \| \dbar^\star_{\psi}(B^{-1}_{\omega_0, \psi}f) \|^2_{(\omega_0, he^{-\psi})} \| u \|^2_{(\omega_0, he^{-\psi})}\\
& \leq \| \dbar^\star_{\psi}(B^{-1}_{\omega_0, \psi}f) \|^2_{(\omega_0, he^{-\psi})} |\llangle B^{-1}_{\omega_0, \psi}f, f \rrangle_{(\omega_0, he^{-\psi})}|. 
\end{align*}
In short, we have $|\llangle B^{-1}_{\omega_0, \psi}f, f \rrangle_{(\omega_0, he^{-\psi})} |\leq \| \dbar^\star_{\psi}(B^{-1}_{\omega_0, \psi}f) \|^2_{(\omega_0, he^{-\psi})}$ for any $\dbar$-closed $f$.
By using the Bochner-Kodaira-Nakano identity $\Delta''_\psi=\Delta'_\psi + [\ai \Theta_{(E, he^{-\psi})}, \Lambda_{\omega_0}]= \Delta'_\psi + A_{1, \omega_0, h}+B_{\omega_0, \psi}$
(cf. \cite[(4.6)]{Dem12}), we get 
\begin{align*}
&\| \dbar^\star_{\psi}(B^{-1}_{\omega_0, \psi}f) \|^2_{(\omega_0, he^{-\psi})} \\
&=\llangle \Delta''_\psi (B^{-1}_{(\omega_0, \psi)}f), B^{-1}_{(\omega_0, \psi)}f\rrangle_{(\omega_0, he^{-\psi})} - \| \dbar(B^{-1}_{(\omega_0, \psi)}f )\|^2_{(\omega_0, he^{-\psi})}  \\
&\leq \llangle \Delta'_\psi (B^{-1}_{(\omega_0, \psi)}f), B^{-1}_{(\omega_0, \psi)}f\rrangle_{(\omega_0, he^{-\psi})} + \llangle A_{1, \omega_0, h} (B^{-1}_{(\omega_0, \psi)}f), B^{-1}_{(\omega_0, \psi)}f\rrangle_{(\omega_0, he^{-\psi})}+ \llangle f, B^{-1}_{(\omega_0, \psi)}f\rrangle_{(\omega_0, he^{-\psi})}.
\end{align*}
Then we obtain 
$$
\llangle A_{1, \omega_0, h} (B^{-1}_{(\omega_0, \psi)}f), B^{-1}_{(\omega_0, \psi)}f\rrangle_{(\omega_0, he^{-\psi})} + \| D'^\star_\psi(B^{-1}_{(\omega_0, \psi)}f )\|^2_{(\omega_0, he^{-\psi})} \geq 0. 
$$

We let $f=\sum_{j, \lambda}f_{j\lambda}dz_1 \wedge \cdots \wedge dz_n \wedge d\bar{z}_j\otimes e_\lambda \in C^\infty_{(n, 1)}(\Delta^n_r, E)$ be a $\dbar$-closed $(n, 1)$-form 
with constant coefficients such that $f(0)=f_0$. 
We can take a positive constant $R\in (0, r)$ such that 
$$
\langle [\ai \Theta_{(E, h)}, \Lambda_{\omega_0}]f, f\rangle_{(\omega_0, h)} = \langle A_{1, \omega_0, h}f, f\rangle_{(\omega_0, h)}<-c
$$
on $\Delta^n_R$ for some positive constant $c>0$. 

Choose a cut-off function $\chi\in \mathscr{D}_{(0, 0)}(\Delta^n_R, \mathbb{R})$ such that $0\leq \chi \leq 1$ and $\chi|_{\Delta^n_{\frac{R}{2}}}\equiv 1$.
We define $v\in \mathscr{D}_{(n, 0)}(\Delta^n_{r}, E)$ by 
$$
v=(-1)^n\sum_{j, \lambda}f_{j\lambda}\bar{z}_j \chi dz_1\wedge \cdots \wedge dz_n\otimes e_\lambda,
$$
and define $g$ by $\dbar v=g$. Then $g\in \mathscr{D}_{(n, 1)}(\Delta^n_r, E)$ and $g=f$ on $\Delta^n_{\frac{R}{2}}$. 
Set $\phi(z)=|z|^2-\frac{R^2}{4}$. Then we have $B_{(\omega_0, m\phi)}=m\cdot $.
We define $\alpha_m:=B^{-1}_{(\omega_0, m\phi)}g=\frac{1}{m}g$. 
Considering the commutation relation $\ai [\Lambda_{\omega_0}, \dbar ]=D'^\star_{m\phi}$ (cf. \cite[(4.5)]{Dem12}), we obtain $D'^\star_{m\phi}\alpha_m=0$ on $\Delta^n_{\frac{R}{2}}$
and $|D'^\star_{m\phi}\alpha_m|_{(\omega_0, h)}\leq \frac{C}{m}$ for some positive constant $C>0$ on $\Delta^n_R \setminus \overline{\Delta}^n_{\frac{R}{2}}$. 
We also have $\langle A_{1, \omega_0, h}\alpha_m, \alpha_m\rangle_{(\omega_0, h)}<-\frac{c}{m^2}$ on $\Delta^n_{\frac{R}{2}}$
and $\langle A_{1, \omega_0, h}\alpha_m, \alpha_m\rangle_{(\omega_0, h)} \leq \frac{C'}{m^2}$ for some $C'>0$ on $\Delta^n_R \setminus \overline{\Delta}^n_{\frac{R}{2}}$
since $g$ has compact support in $\Delta^n_{R}$. Set $C'':=C^2+C'$. To summarize, we obtain 
\begin{align*}
0&\leq \llangle A_{1, \omega_0, h} (B^{-1}_{(\omega_0, m\phi)}g), B^{-1}_{(\omega_0, m\phi)}g\rrangle_{(\omega_0, he^{-m\phi})} + \| D'^\star_{m\phi}(B^{-1}_{(\omega_0, m\phi)}g)\|^2_{(\omega_0, he^{-m\phi})}\\
&=\llangle A_{1, \omega_0, h} \alpha_m, \alpha_m \rrangle_{(\omega_0, he^{-m\phi})} + \| D'^\star_{m\phi}\alpha_m \|^2_{(\omega_0, he^{-m\phi})}\\
&=\int_{\Delta^n_{\frac{R}{2}}}\langle A_{1, \omega_0, h}\alpha_m, \alpha_m\rangle_{(\omega_0, h)}e^{-m\phi}dV_{\omega_0}+\int_{\Delta^n_R \setminus \overline{\Delta}^n_{\frac{R}{2}}}\langle A_{1, \omega_0, h}\alpha_m, \alpha_m\rangle_{(\omega_0, h)}e^{-m\phi}dV_{\omega_0}\\
& \ \ + \int_{\Delta^n_R \setminus \overline{\Delta}^n_{\frac{R}{2}}}|D'^\star_{m\phi}\alpha_m|^2_{(\omega_0, h)}e^{-m\phi}dV_{\omega_0}\\
& \leq -\frac{c}{m^2}\int_{\Delta^n_{\frac{R}{2}}}e^{-m\phi}dV_{\omega_0} + \frac{C''}{m^2}\int_{\Delta^n_R \setminus \overline{\Delta}^n_{\frac{R}{2}}}e^{-m\phi}dV_{\omega_0}
\end{align*}
for any $m\in \mathbb{N}$. Hence, we get 
$$
-c\int_{\Delta^n_{\frac{R}{2}}}e^{-m\phi}dV_{\omega_0} + C''\int_{\Delta^n_R \setminus \overline{\Delta}^n_{\frac{R}{2}}}e^{-m\phi}dV_{\omega_0}\geq 0.
$$
Since $\phi<0$ on $\Delta^n_{\frac{R}{2}}$ and $\phi>0 $ on $\Delta^n_R \setminus \overline{\Delta}^n_{\frac{R}{2}}$, the first term has a negative upper bound which is independent of $m$
$$
-c\int_{\Delta^n_{\frac{R}{2}}}e^{-m\phi}dV_{\omega_0}<-c |\Delta^n_{\frac{R}{2}}|.
$$
The second term goes to zero as $m\to +\infty$ by Lebesgue's dominated convergence theorem. 
Then for sufficiently large $m>\!>1$, we have 
$$
-c\int_{\Delta^n_{\frac{R}{2}}}e^{-m\phi}dV_{\omega_0} + C''\int_{\Delta^n_R \setminus \overline{\Delta}^n_{\frac{R}{2}}}e^{-m\phi}dV_{\omega_0}< 0,
$$
which is a contradiction. Consequently, we can conclude that $h$ is Nakano semi-positive on $\Delta^n_r$. 
\end{proof}

\subsection{Singular Hermitian metrics}
In this subsection, we consider the case that a Hermitian metric has singularities. 
First, we introduce the definition of singular Hermitian metrics on vector bundles. 

\begin{definition} $($\cite[Section 3]{BP08}, \cite[Definition 17.1]{HPS18}, \cite[Definition 2.2.1]{PT18} and \cite[Definition 1.1]{Rau15}$)$\label{def:shm}
We say that $h$ is a {\it singular Hermitian metric}
on $E$ if $h$ is a measurable map from the base manifold $X$ to
the space of non-negative Hermitian forms on the fibers satisfying $0< {\rm det} h < +\infty $ almost everywhere.
\end{definition}

Related to the notion of singular Hermitian metrics, we introduce the ideal sheaves. 

\begin{definition}$($\cite{Nad90}$)$
Let $h$ be a singular Hermitian metric on a holomorphic line bundle $L\to X$, and $\varphi$ be the local weight of $h$, i.e. $h=e^{-\varphi}$ locally. 
Then we define the ideal subsheaf $\mathscr{I}(h)\subset \mathscr{O}_X$ of germs of holomorphic functions as follows{\rm :}
$$
\mathscr{I}(h)_x := \{ f_x\in\mathscr{O}_{X, x} \mid  |f_x|^2e^{-\varphi} \text{~is locally integrable around~} x \}.
$$
\end{definition}
We can easily verify that the above definition is independent of the choice of local weights.
In \cite{Nad90}, Nadel proved that $\mathscr{I}(h)$ is coherent by using the H\"ormander $L^2$-estimate. 
We can also define a higher-rank analogue of the multiplier ideal sheaf $\mathscr{I}(h)$. 

\begin{definition}$($cf. \cite{deC98}$)$
Let $h$ be a singular Hermitian metric on a holomorphic vector bundle $E\to X$. 
We define the ideal subsheaf $\mathscr{E}(h)$ of germs of local holomorphic sections of $E$ as follows{\rm :}
$$
\mathscr{E}(h)_x:= \{ s_x \in \mathscr{O}(E)_x \mid |s_x|^2_h \text{~is locally integrable around~}x\}.
$$
\end{definition}
In \cite{HI19}, we prove that $\mathscr{E}(h)$ is coherent if $h$ satisfies the twisted H\"ormander condition above. 
We can also show that $\mathscr{E}(h)$ is coherent when $h$ is a Nakano semi-positive singular Hermitian metric (cf. Proposition \ref{prop:coherence}). 

The Chern curvature tensor $\Theta_{(E, h)}$ of a smooth Hermitian metric $h$ can be locally defined by $\dbar (h^{-1}\partial h)$. 
On a holomorphic line bundle, the Chern curvature of a positive or negative singular Hermitian metric can be also defined in the sense of currents. 
However, for a holomorphic vector bundle $E$ of rank $E\geq 2$, it is not possible to define the Chern curvature current with measure coefficients in general. 
This phenomenon was observed by Raufi in \cite{Rau15}. 
Before showing the example, we introduce the definitions of Griffiths semi-negativity and Griffiths semi-positivity. 
\begin{definition}$($\cite[Definition 3.1]{BP08}, \cite[Definition 2.2.2]{PT18} and \cite[Definition 1.2]{Rau15}$)$\label{GrifPos}
We say that a singular Hermitian metric $h$ is: 
\begin{enumerate}
\item {\it Griffiths semi-negative} if $|u|_h$ is plurisubharmonic for any local holomorphic section $u\in \mathscr{O}(E)$ of $E$. 
\item {\it Griffiths semi-positive} if the dual metric $h^\star$ on $E^\star$ is Griffiths semi-positive. 
\end{enumerate}
\end{definition}

This definition arises from a characterization of Griffiths semi-positivity (see Proposition \ref{prop:grif}).
Then Raufi found the following example. 
\begin{theorem}$($\cite[Theorem 1.5]{Rau15}$)$
Let $E$ be the trivial vector bundle $\Delta\times \mathbb{C}^2$ over $\Delta:=\Delta^1_1\subset \mathbb{C}$. 
Let $h$ be the singular Hermitian metric 
$$
h= \left(
\begin{array}{cc}
1+|z|^2 & z \\
\bar{z} & |z|^2
\end{array}
\right) . 
$$ 
Then $h$ is Griffiths semi-negative, and 
$\Theta_{(E, h)}$ is not a current with measure coefficients.
\end{theorem}

This result implies that we cannot define the positivity or negativity by using the Chern curvature currents. 
Furthermore, the strict positivity or negativity is not generally formulated. 
If there is a K\"ahler metric on $X$, we can define the strict Griffiths positivity as follows. 

\begin{definition}$($\cite[Definition 2.6]{Ina18}$)$\label{def:strictgrif}
Let $\omega_X$ be a K\"ahler metric on $X$. 
We say that a singular Hermitian metric $h$ is {\it strictly Griffiths $\delta_{\omega_X}$-positive} if for any open subset $U$ and for any K\"ahler potential $\varphi$ of $\omega_X$ on $U$, i.e. $\ai \ddbar \varphi=\omega_X$ on $U$, 
$he^{\delta\varphi}$ is Griffiths semi-positive on $U$.
%
\end{definition}


For Nakano semi-positivity of singular Hermitian metrics, we can characterize it by using Proposition \ref{prop:nakano} (see Definition \ref{def:nakanosemipositive}). 
We can also define the strict Nakano $\delta_{\omega_X}$-positivity of singular Hermitian metrics as follows. 
\begin{definition}\label{def:strictnakano}
Let $(X, \omega_X)$ be a K\"aher manifold. 
We say that $h$ is {\it globally (resp. locally) strictly Nakano $\delta_{\omega_X}$-positive} if for any open subset $U$ and for any K\"ahler potential $\varphi$ of $\omega_X$, i.e. $\ai \ddbar \varphi = \omega_X$ on $U$, 
$he^{\delta \varphi}$ is globally (resp. locally) Nakano semi-positive on $U$ in the sense of Definition \ref{def:nakanosemipositive} (resp. Definition \ref{def:localnakano}). 
%
\end{definition}

\begin{remark}\label{rem:2k}
We consider the following condition related to the condition $(2)$ in Proposition \ref{prop:nakano} for $k\geq 1$. 

(2-k): For any Stein coordinate $(\Omega, \iota)$ such that $E|_{\iota (\Omega)}$ is trivial on $\iota( \Omega)$, 
for any K\"ahler form $\omega_\Omega$ on $\Omega$, 
for any smooth strictly plurisubharmonic function $\psi$ on $\Omega$, for any positive integer $q$ such that $1\leq q\leq k$, 
and for any $\dbar$-closed $f\in L^2_{(n, q)}(\Omega, \iota^\star E ; \omega_\Omega, (\iota^\star h)e^{-\psi})$, 
there exists $u\in  L^2_{(n, q-1)}(\Omega, \iota^\star E ; \omega_\Omega, (\iota^\star h)e^{-\psi})$ satisfying $\dbar u =f$ and 
$$
\int_\Omega |u|^2_{(\omega_\Omega, \iota^\star h)}e^{-\psi}dV_{\omega_\Omega} \leq \int_\Omega \langle B^{-1}_{\omega_\Omega, \psi}f, f \rangle_{(\omega_\Omega, \iota^\star h)}e^{-\psi} dV_{\omega_\Omega}, 
$$
provided the right-hand side is finite. 

The proof of Proposition \ref{prop:nakano} suggests that we only have to consider all $(n, 1)$-forms $f$, not all $(n, q)$-forms for $1\leq q \leq n$.
However, the conditions (2-1),$\cdots$,(2-n) are equivalent to each other under the assumption that $h$ is smooth.
Hence, in this paper, we adopt the seemingly stronger condition (2-n) (=Definition \ref{def:nakanosemipositive}) to define global Nakano semi-positivity of singular Hermitian metrics. 
Related to this remark, we propose Question \ref{ques:2k} in Section \ref{sec:problems}.
\end{remark}

\subsection{Big vector bundles}
In this subsection, to prove Theorem \ref{thm:bigvanishing}, we prepare notions of big vector bundles. 
We let $X$ be a projective manifold. 

\begin{definition}$(${\cite[Section 2]{BKKMSU15}}$)$\label{def:baselocus}
	We define the {\it base locus} of $E$ as the set 
	$$
	{\rm Bs}(E):= \{x\in X \mid  H^0(X, E)\to E_x \text{ is not surjective}\},
	$$
	and the {\it stable base locus} of $E$ as the set 
	$$
	\mathbb{B}(E):=\bigcap_{m\in \mathbb{N}}{\rm Bs}(S^mE).
	$$ 
For an ample line bundle $A$, we also define the {\it augmented base locus} of $E$ by 
$$
\mathbb{B}^{A}_+(E):= \bigcap_{p/q\in \mathbb{Q}}\mathbb{B}(S^qE\otimes A^{-p}).
$$
Note that $\mathbb{B}^{A}_+(E)$ does not depend on the choice of the ample line bundle. 
Hence, we write $\mathbb{B}_+(E)$ for simplicity. 
\end{definition}

With these notations, we introduce the following definitions. 
We let $\pi: \mathbb{P}(E)\to X$ denote the projective bundle of rank one quotients of $E$,
and $\mathscr{O}_{\mathbb{P}(E)}(1)$ denote the universal quotient of $\pi^\star E$ over $\mathbb{P}(E)$.

\begin{definition}$($\cite[Theorem 1.1, Definition 6.1]{BKKMSU15}$)$\label{def:lbigvbig}
	We say that 
	\begin{enumerate}
		\item $E$ is {\it L-big} if $\mathscr{O}_{\mathbb{P}(E)}(1)$ is big on $\mathbb{P}(E)$.
		\item $E$ is {\it V-big} (or {\it Viehweg-big}) if $\mathbb{B}_+(E)\neq X$. 
	\end{enumerate}
\end{definition}

We remark that if $E$ is V-big, then $E$ is L-big as well \cite[Corollary 6.5]{BKKMSU15}. 
In order to prove Theorem \ref{thm:bigvanishing}, we need the following proposition. 

\begin{proposition}$($\cite[Proposition 3.2]{BKKMSU15}$)$\label{prop:baseloci}
	Keep the notation. Then 
	$$
	\pi(\mathbb{B}_+(\mathscr{O}_{\mathbb{P}(E)}(1)))=\mathbb{B}_+(E). 
	$$
\end{proposition}

\section{Demailly and Skoda's theorem in the singular setting}\label{sec:dsk}
In this section, we prove Theorem \ref{thm:demaillyskoda}, which is a generalization of Demailly and Skoda's result. 
Before proving that, we explain Demailly and Skoda's theorem. 
\begin{theorem}$($\cite{DS}$)$\label{thm:smoothds}
Let $h$ be a smooth Hermitian metric on $E$. If $(E, h)$ is Griffiths semi-positive, then $(E\otimes \det E, h\otimes \det h)$ is Nakano semi-positive. 
\end{theorem}

Taking a smooth approximating sequence $\{ h_\nu \}_{\nu =1}^\infty$ of $h$, we give a proof of Theorem \ref{thm:demaillyskoda}. 
Our main approximation technique is based on the following proposition obtained by Berndtsson and Paun. 

\begin{proposition}$($cf. \cite[Proposition 3.1]{BP08}, \cite{Rau15}$)$\label{prop:bp}
Let $E$ be a trivial vector bundle over a polydisc $U$ and $h$ be a Griffiths semi-positive singular Hermitian metric on $E$. 
Then there exists a sequence of smooth Hermitian metrics $\{ h_\nu \}_{\nu =1}^\infty$, with positive Griffiths curvature, increasing to $h$ on smaller polydiscs. 
\end{proposition}

We remark that the above proposition is valid if $U$ is not a polydisc but a domain.
A sequence of smooth Hermitian metrics approximating $h$ is obtained through convolution of $h$ with an approximate identity. 
In this way, we can only get an approximating sequence when $E$ is a trivial vector bundle over a domain in $\mathbb{C}^n$.

To prove Theorem \ref{thm:demaillyskoda}, we also need the following theorem. 
\begin{theorem}$($\cite[Corollary 1]{Siu76}$)$\label{thm:siu}
Let $X$ be a Stein submanifold of $\mathbb{C}^N$ for some $N>n=\dim X$. 
Let $i:X\to \mathbb{C}^N$ be an inclusion map. 
Then there exists an open neighborhood $U$ of $X$ in $\mathbb{C}^N$ such that $U$ is a holomorphic retraction of $X$, 
i.e. there exists a holomorphic map $p:U\to X$ such that $p\circ i = id_X$. 
\end{theorem}

Then we give a proof of the following result. 

\begin{theorem}$($= {\rm Theorem \ref{thm:demaillyskoda}}$)$\label{thm:newds}
Let $h$ be a singular Hermitian metric on $E$. If $(E, h)$ is Griffiths semi-positive, then $(E\otimes \det E, h\otimes \det h)$ is globally Nakano semi-positive in the sense of singular Hermitian metrics. 
\end{theorem}

\begin{proof}
It is clear that Griffiths semi-positivity of $h$ yields the Griffiths semi-positivity of $h \otimes \det h$ (cf. \cite[Proposition 1.3]{Rau15}).
Then 
it is enough to show that $(E\otimes \det E, h\otimes \det h)$ satisfies the condition in Definition \ref{def:nakanosemipositive}. 

Let $(\Omega, \iota)$ be an arbitrary Stein coordinate of $X$ such that $(E\otimes \det E)|_{\iota(\Omega)}$ is trivial on $\iota(\Omega)$. 
Since $\Omega$ can be properly embedded into $\mathbb{C}^N$ for some large $N$, we can regard $\Omega$ as a submanifold of $\mathbb{C}^N$ without any loss of generality. 
From Theorem \ref{thm:siu}, we take an open neighborhood $U$ of $\Omega$ in $\mathbb{C}^N$ and a holomorphic map $p : U\to \Omega$ which defines a holomorphic retraction of $\Omega$, i.e. $p\circ i=id_\Omega$, where $i: \Omega \to \mathbb{C}^N$ is an inclusion map. 
Since $(E\otimes \det E)|_{\iota(\Omega)}$ is a trivial bundle, $\iota^\star(E\otimes \det E)$ and $p^\star \iota^\star(E\otimes \det E)$ are also trivial on $\Omega$ and $U$. 
Thanks to \cite[Lemma 2.3.2]{PT18}, $\iota^\star h$ and $p^\star \iota^\star h$ are also Griffiths semi-positive. 
For the sake of clarity, we omit the map $\iota$ and simply write $(E, h)(=(\iota^\star E, \iota^\star h))$ on $\Omega$.

Since $E\otimes \det E$ is trivial on $\Omega$, we fix a holomorphic global frame $(e_1, \cdots , e_r)$ of $E\otimes \det E$ on $\Omega$. 
Then $(\det (E\otimes \det E), \det (h\otimes \det h)) \cong ((\det E)^{\otimes r+1}, (\det h)^{\otimes r+1})$ is also trivial on $\Omega$ with respect to the frame $e_1\wedge \cdots \wedge e_r$. 
We define the function $\Psi$ by 
$$
|e_1 \wedge \cdots \wedge e_r|_{(\det h)^{\otimes r+1}}=e^{-\Psi}.
$$
Since $(\det h)^{\otimes r+1}$ is Griffiths semi-positive (cf. \cite[Proposition 1.3]{Rau15}), $\Psi$ is a plurisubharmonic function on $\Omega$. 
We construct the metric $h\otimes \det h~ e^{\frac{\Psi}{r+1}}$ on $E\otimes \det E$. 
We can easily see that $h\otimes \det h~ e^{\frac{\Psi}{r+1}}$ is Griffiths semi-positive (for the detailed proof, see Proposition \ref{prop:detail} below).
From Proposition \ref{prop:bp}, we get a sequence of smooth Hermitian metrics $\{ h_\nu \}^\infty_{\nu=1}$, with positive Griffiths curvature, increasing to $p^\star(h\otimes \det h~ e^{\frac{\Psi}{r+1}})$ on $p^\star (E\otimes \det E)$ over any relatively compact subdomain of $U$. 
Set $g_\nu := i^\star h_\nu$. 
Since $p\circ i =id_\Omega$, $\{ g_\nu \}_{\nu=1}^\infty$ is also a sequence of smooth Hermitian metrics, with positive Griffiths curvature, increasing to $h\otimes \det h~ e^{\frac{\Psi}{r+1}}$ on $E\otimes \det E$ over any relatively compact subset of $\Omega$. 
We also have that $\{ \det g_\nu \}_{\nu=1}^\infty$ becomes a sequence of smooth Hermitian metrics, with positive curvature, increasing to 
\begin{align*}
(\det (E\otimes \det E), \det (h\otimes \det h~ e^{\frac{\Psi}{r+1}}))&=((\det E)^{\otimes r+1}, (\det h)^{\otimes r+1}e^{\frac{r\Psi}{r+1}})\\
&\cong (\mathbb{C}, e^{-\frac{\Psi}{r+1}})
\end{align*}
(cf. \cite[the proof of Proposition 1.3]{Rau15}). 
Then, from the result of Demailly-Skoda (Theorem \ref{thm:smoothds}), $\{ g_\nu \otimes \det g_\nu \}_{\nu=1}^\infty$ gives a sequence of smooth Hermitian metrics, with positive Nakano curvature, increasing to $h\otimes \det h$ on $E\otimes \det E$ over any relatively compact subset of $\Omega$.
Here we regard $g_\nu \otimes \det g_\nu $ as the metric on $E\otimes \det E$ via the trivialization of $(\det E)^{\otimes r+1}$ for every $\nu \in \mathbb{N}$. 

Then we take an arbitrary K\"ahler metric $\omega_\Omega$, an arbitrary smooth strictly plurisubharmonic function $\psi$, and an arbitrary $\dbar$-closed $f\in L^2_{(n, q)}(\Omega, E\otimes \det E ; \omega_\Omega, h\otimes \det h e^{-\psi})$ for any $q>0$ on $\Omega$. 
We also take a Stein exhaustion $\{ \Omega_j \}_{j=1}^\infty$ of $\Omega$, where $\Omega_j$ is a relatively compact Stein subdomain. 
We assume that 
$$
\int_\Omega \langle B^{-1}_{\omega_\Omega, \psi}f, f\rangle_{(\omega_\Omega, h\otimes \det h)}e^{-\psi} dV_{\omega_\Omega} < +\infty. 
$$
Since $\{ g_\nu \otimes \det g_\nu \}_{\nu=1}^\infty$ is an increasing sequence on any relatively compact subset, 
we have 
$$
\int_{\Omega_j} \langle B^{-1}_{\omega_\Omega, \psi}f, f\rangle_{(\omega_\Omega, g_\nu \otimes \det g_\nu)}e^{-\psi} dV_{\omega_\Omega} < +\infty 
$$
for fixed $j \in \mathbb{N}$. 
Thanks to H\"ormander's $L^2$-estimate for smooth Hermitian metrics (cf. Theorem \ref{thm:demailly}) and the proof of Proposition \ref{prop:nakano}, we get a solution $u_\nu \in L^2_{(n, q-1)}(\Omega_j, E\otimes \det E ; \omega_\Omega, g_\nu \otimes \det g_\nu e^{-\psi})$ of $\dbar u_\nu =g$ such that 
\begin{align*}
\int_{\Omega_j} |u_\nu|^2_{(\omega_\Omega, g_\nu\otimes \det g_\nu)}e^{-\psi}dV_{\omega_\Omega} & \leq \int_{\Omega_j} \langle A^{-1}_{q, \omega_\Omega, g_\nu\otimes \det g_\nu e^{-\psi}}f, f \rangle_{(\omega_\Omega, g_\nu\otimes \det g_\nu)} e^{-\psi} dV_{\omega_\Omega}\\
&\leq \int_{\Omega_j} \langle B^{-1}_{\omega_\Omega, \psi}f, f\rangle_{(\omega_\Omega, g_\nu\otimes \det g_\nu)} e^{-\psi} dV_{\omega_\Omega} \\
&\leq \int_{\Omega} \langle B^{-1}_{\omega_\Omega, \psi}f, f\rangle_{(\omega_\Omega, h\otimes \det h)}e^{-\psi} dV_{\omega_\Omega} < +\infty 
\end{align*}
since $g_\nu \otimes \det g_\nu$ is Nakano semi-positive. 
For fixed $\nu_0 $,  $\{ u_\nu \}_{\nu\geq \nu_0}$ forms a bounded sequence in $L^2_{(n, q-1)}(\Omega_j, E\otimes \det E ; \omega_\Omega, g_{\nu_0}\otimes \det g_{\nu_0}e^{-\psi})$
due to the monotonicity of $\{ g_\nu \otimes \det g_\nu \}_{\nu=1}^\infty$. 
Hence, we can obtain a weakly convergent subsequence in $L^2_{(n, q-1)}(\Omega_j, E\otimes \det E ; \omega_\Omega , g_{\nu_0}\otimes \det g_{\nu_0}e^{-\psi})$. 
By using a diagonal argument, we get a subsequence $\{ u_{\nu_k}\}_{k=1}^\infty$ of $\{ u_\nu \}_{\nu=1}^\infty$ converging weakly in $L^2_{(n, q-1)}(\Omega_j, E\otimes \det E ; \omega_\Omega , g_{\nu_0}\otimes \det g_{\nu_0}e^{-\psi})$ for any $\nu_0$. 
We denote by $u_j$ the weak limit of $\{ u_{\nu_k}\}_{k=1}^\infty$. 
Then $u_j$ satisfies $\dbar u_j= f$ on $\Omega_j$ and 
$$
\int_{\Omega_j} |u_j|^2_{(\omega_\Omega, g_{\nu_0}\otimes \det g_{\nu_0})}e^{-\psi}dV_{\omega_\Omega}  \leq \int_{\Omega} \langle B^{-1}_{\omega_\Omega, \psi}f, f\rangle_{(\omega_\Omega, h\otimes \det h)}e^{-\psi} dV_{\omega_\Omega}
$$
for each $\nu_0$. 
Taking weak limits $\nu_0 \to + \infty$ and using the monotone convergence theorem, we have the following estimate
$$
\int_{\Omega_j} |u_j|^2_{(\omega_\Omega, h\otimes \det h)}e^{-\psi}dV_{\omega_\Omega}  \leq \int_{\Omega} \langle B^{-1}_{\omega_\Omega, \psi}f, f\rangle_{(\omega_\Omega, h\otimes \det h)}e^{-\psi} dV_{\omega_\Omega}.
$$
Repeating the above argument and taking the weak limit $j\to \infty$, we get a solution $u\in L^2_{(n, q-1)}(\Omega, E\otimes \det E ; \omega_\Omega, h\otimes \det he^{-\psi})$ of $\dbar u =f$ such that 
$$
\int_{\Omega} |u|^2_{(\omega_\Omega, h\otimes \det h)}e^{-\psi}dV_{\omega_\Omega}  \leq \int_{\Omega} \langle B^{-1}_{\omega_\Omega, \psi}f, f\rangle_{(\omega_\Omega, h\otimes \det h)}e^{-\psi} dV_{\omega_\Omega}
$$
on $\Omega$. Consequently, we can conclude that $h\otimes \det h$ is Nakano semi-positive in the sense of singular Hermitian metrics. 
\end{proof}

\begin{proposition}\label{prop:detail}
Let notation be the same as one in the proof of Theorem \ref{thm:newds}. 
Then the metric $h\otimes \det h~ e^{\frac{\Psi}{r+1}}$ is Griffiths semi-positive on $E\otimes \det E$. 
\end{proposition}

\begin{proof}
We have to show that $\log |u|_{h^\star \otimes \det h^\star e^{-\frac{\Psi}{r+1}}}$ is plurisubharmonic for any local holomorphic section $u\in \mathscr{O}(E^\star \otimes \det E^\star)$ of $E^\star \otimes \det E^\star$.
Let $(e^\star_1, \cdots , e^\star_r)$ be the global dual frame of $(e_1, \cdots , e_r)$. 
We also take a local frame of $(\epsilon_1, \cdots, \epsilon_r)$ of $E$ and let $(\epsilon^\star_1, \cdots, \epsilon^\star_r)$ the local dual frame.
Fixing these frames, it is enough to show that 
$$
\log (|u|_{h^\star}|\epsilon^\star_1 \wedge \cdots \wedge \epsilon^\star_r|_{\det h^\star}e^{-\frac{\Psi}{r+1}})= \log |u|_{h^\star} + \log |\epsilon^\star_1 \wedge \cdots \wedge \epsilon^\star_r|_{\det h^\star}|e_1\wedge \cdots \wedge e_r|^{\frac{1}{r+1}}_{(\det h)^{\otimes r+1}}
$$
is plurisubharmonic. Since $h^\star$ is Griffiths semi-negative, $\log |u|_{h^\star}$ is a plurisubharmonic function. 
We define a local holomorphic function $f$ by $f(\epsilon^\star_1 \wedge \cdots \wedge \epsilon^\star_r)^{\otimes r+1}=e^\star_1\wedge \cdots \wedge e^\star_r$. 
Then we obtain 
\begin{align*}
(r+1)\log |\epsilon^\star_1 \wedge \cdots \wedge \epsilon^\star_r|_{\det h^\star}|e_1\wedge \cdots \wedge e_r|^{\frac{1}{r+1}}_{(\det h)^{\otimes r+1}}&=\log |\epsilon^\star_1 \wedge \cdots \wedge \epsilon^\star_r|^{r+1}_{\det h^\star}|e_1\wedge \cdots \wedge e_r|_{(\det h)^{\otimes r+1}}\\
&=\log \left( \frac{|(\epsilon^\star_1 \wedge \cdots \wedge \epsilon^\star_r)^{r+1}|_{(\det h^\star)^{\otimes r+1}}}{|e^\star_1\wedge \cdots \wedge e^\star_r|_{(\det h^\star)^{\otimes r+1}}}\right)\\
&= \log |f|.
\end{align*}
Since $f\neq 0$, this term is a harmonic function. Therefore, we complete the proof. 
\end{proof}

If $X$ admits a K\"ahler metric $\omega_X$, we can also prove the following theorem. 

\begin{theorem}\label{thm:dsstrict}
Let $\omega_X$ be a K\"ahler form on a K\"ahler manifold $X$. 
If $(E, h)$ is strictly Griffiths $\delta_{\omega_X}$-positive, then $(E\otimes \det E, h\otimes \det h)$ is strictly Nakano $(r+1)\delta_{\omega_X}$-positive.
\end{theorem}

\begin{proof}
We take an arbitrary open subset $U$ and any K\"ahler potential $\varphi$ of $\omega_X$ on $U$. 
We also take a Stein coordinate $(\Omega, \iota)$ of $U$. 
Then we use the same notation as in the proof of Theorem \ref{thm:newds}. 
By the definition of the strict Griffiths $\delta_{\omega_X}$-positivity, we have that $he^{\delta \varphi}$ is Griffiths semi-positive.
Hence, from Theorem \ref{thm:demaillyskoda}, we get 
$$
he^{\delta \varphi}\otimes \det (he^{\delta \varphi})=h\otimes \det h e^{(r+1)\delta \varphi}
$$
is globally Nakano semi-positive in the sense of singular Hermitian metrics on $U$.
Thus we can conclude that $h\otimes \det h$ is strictly Nakano $(r+1)\delta_{\omega_X}$-positive on $X$.  
\end{proof}

\section{$L^2$-estimates and vanishing theorems}\label{sec:vanishing}
In this section, we give an $L^2$-estimate and a vanishing theorem for holomorphic vector bundles with strictly Nakano positive singular Hermitian metrics. 
Then we prove Theorem \ref{thm:l2}, \ref{thm:vanishing}, and \ref{cor:grif}. 
In this section, we assume that $X$ is a projective manifold and $\omega_X$ is a K\"ahler form on $X$. 
First of all, we show Theorem \ref{thm:l2}.

\begin{proof}[\indent\sc Proof of Theorem \ref{thm:l2}]
Choose an arbitrary $\dbar$-closed $f\in L^2_{(n, q)}(X, E ; \omega_X, h)$ for $q>0$.
By Serre's GAGA \cite{Ser56}, there exists a proper Zariski open subset $Z\neq \emptyset$ such that $E|_Z$ is trivial over $Z$.  
We can also take $Z$ such that $Z$ is Stein and $\omega_X$ is $\ddbar$-exact on $Z$. 
Then $(Z, i)$ is a Stein coordinate of $X$ such that $E|_Z$ is trivial on $Z$, where $i:Z\to X$ is the natural inclusion map. 
We fix a K\"ahler potential $\varphi$ of $\omega_X$ on $Z$, i.e. $\varphi$ satisfies $\ai \ddbar \varphi =\omega_X$. 
Then we have that 
\begin{align*}
\langle [B_{\omega_X, \delta\varphi}, \Lambda_{\omega_X}]f, f\rangle_{(\omega_X, h)} &= \delta q |f|^2_{(\omega_X, h)},\\
\langle [B^{-1}_{\omega_X, \delta\varphi}, \Lambda_{\omega_X}]f, f\rangle_{(\omega_X, h)} &= \frac{1}{\delta q} |f|^2_{(\omega_X, h)},
\end{align*}
respectively. 

Thanks to the definition of the strict Nakano $\delta_{\omega_X}$-positivity, for any smooth strictly plurisubharmonic function $\psi$ on $Z$, 
we can obtain $u\in L^2_{(n, q-1)}(Z, E ; \omega_X, he^{\delta\varphi -\psi})$
satisfying $\dbar u =f$ and 
$$
\int_Z |u|^2_{(\omega_X, h)}e^{\delta\varphi -\psi}dV_{\omega_X} \leq \int_Z \langle B^{-1}_{\omega_X, \psi}f, f \rangle_{(\omega_X, h)}e^{\delta\varphi -\psi}dV_{\omega_X}
$$
if the right-hand side is finite. Taking $\psi =\delta \varphi$, we get a solution $u\in L^2_{(n, q-1)}(Z, E ; \omega_X, h)$ of $\dbar u =f $ such that 
\begin{align*}
\int_Z |u|^2_{(\omega_X, h)}dV_{\omega_X} &\leq \int_Z \langle B^{-1}_{\omega_X, \delta\varphi}f, f \rangle_{(\omega_X, h)}dV_{\omega_X}\\
& = \frac{1}{\delta q} \int_Z |f|^2_{(\omega_X, h)}dV_{\omega_X}\\
& \leq \frac{1}{\delta q} \int_X |f|^2_{(\omega_X, h)}dV_{\omega_X}< +\infty. 
\end{align*}
Letting $u=0$ on $X\setminus Z$, we have $u\in L^2_{(n, q-1)}(X, E ; \omega_X, h)$, $\dbar u=f$, and 
$$
\int_X |u|^2_{(\omega_X, h)}dV_{\omega_X}  \leq \frac{1}{\delta q} \int_X |f|^2_{(\omega_X, h)}dV_{\omega_X}
$$
from the following lemma. 

\end{proof}

\begin{lemma}$($cf. \cite[Lemma 5.1.3]{Ber10}$)$\label{lem:things}
Let $X$ be a complex manifold and let $S$ be a complex hypersurface in $X$.
Let $u$ and $f$ be (possibly bundle valued) forms in $L^2_{loc}$ of $X$ satisfying $\dbar u =f$ on $X\setminus S$. 
Then the same equation holds on $X$ (in the sense of distributions). 
\end{lemma}

\begin{remark}
Lemma \ref{lem:things} holds when $h$ is smooth. 
However, since we assume that $h$ is Griffiths semi-positive, we can locally take a sequence of smooth Hermitian metrics increasing to $h$ from Proposition \ref{prop:bp}.
Thus, we have that $f$ and $u$ are $L^2_{loc}$ forms with respect to some smooth Hermitian metric.
Therefore, we can apply Lemma \ref{lem:things}.
\end{remark}

By using Theorem \ref{thm:l2}, we prove Theorem \ref{thm:vanishing}. 
Before proving Theorem \ref{thm:vanishing}, we state the following vanishing theorem for holomorphic line bundles, which was obtained by Nadel in \cite{Nad90} and generalized by Demailly in \cite{Dem93}.

\begin{theorem}$($\cite{Nad90}, \cite{Dem93}, and \cite[(5.11)]{Dem12}$)$
Let $(X, \omega_X)$ be a K\"ahler weakly pseudoconvex manifold, and $L\to X$ be a holomorphic line bundle equipped with a singular Hermitian metric $h$ of weight $\varphi$. 
We assume that $\ai \Theta_{(L, h)}\geq \epsilon \omega$ for some continuous positive function $\epsilon$ on $X$. 
Then 
$$
H^q(X, K_X\otimes L \otimes \mathscr{I}(h))=0
$$
for $q>0$. 
\end{theorem}

We also mention the following result related to the coherence of $\mathscr{E}(h)$. 

\begin{proposition}$($cf. \cite[Theorem 1.4]{HI19}$)$\label{prop:coherence}
Let $h$ be a globally (or only locally) Nakano semi-positive singular Hermitian metric and $\mathscr{E}(h)$ be the sheaf of germs of locally square integrable holomorphic sections of $E$ with respect to $h$. 
Then $\mathscr{E}(h)$ is a coherent subsheaf of $\mathscr{O}(E)$. 
\end{proposition}

In the paper \cite{HI19}, the authors prove Proposition \ref{prop:coherence} in the case that $h$ is positively curved in the sense of twisted H\"ormander. 
Although the twisted H\"ormander condition (cf. Definition \ref{def:twistedhorm}) is slightly different from the definition of singular Nakano semi-positivity, 
the proof of Proposition \ref{prop:coherence} is almost the same as the proof in \cite{HI19}. 
Hence, we only mention a sketch of the proof here for the sake of clarity. 

\begin{proof}[\indent\sc Proof of Proposition \ref{prop:coherence}]
	Since the result is local, we fix an arbitrary polydisc $\Delta\subset X$ which trivializes $E=\underline{\mathbb{C}}^r$. 
	Fix a coordinate $(z_1, \ldots , z_n)$ on $\Delta$. 
	Let $H^0_{(2, h)}(\Delta, \underline{\mathbb{C}}^r)$ be the space of the square integrable $\underline{\mathbb{C}}^r$-valued holomorphic functions with respect to $h$ on $\Delta$. 
	$H^0_{(2, h)}(\Delta, \underline{\mathbb{C}}^r)$ generates a coherent ideal sheaf $\mathscr{F}\subset \mathscr{O}(\underline{\mathbb{C}}^r)$. 
	First, we will show that 
	$$
	\mathscr{E}(h)_x \subset \mathscr{F}_x + \mathscr{E}(h)_x\cap \mathfrak{m}_x^{k+1}\cdot \mathscr{O}(\underline{\mathbb{C}}^r)_x
	$$
	for any $k\in \mathbb{N}$, where $x\in \Delta$ and $\mathfrak{m}_x$ is a maximal ideal of $\mathscr{O}(\underline{\mathbb{C}}^r)_x$. 
	Take an element $f={}^t(f_{1,x}, \ldots , f_{r, x})\in \mathscr{E}(h)_x$. 
	Let $\theta$ be a cut-off function around $x$. 
	We consider a $\dbar$-closed $\underline{\mathbb{C}}^r$-valued $(n,1)$-form $\alpha=\dbar(\theta f dz)$. 
	We also take a smooth strictly plurisubharmonic function $\psi_{\delta}(z)=(n+k)\log(|z-x|^2+\delta^2)+|z|^2$. 
	By definition of the global Nakano semi-positivity of $(E, h)$, we can solve $\dbar$-equations with the estimate of $L^2$-norms on $\Delta$. 
	Then we get solutions $\{ u_\delta\}_\delta$ satisfying $\dbar u_\delta=\alpha$ and the $L^2$-estimates with respect to the weight $\psi_\delta$. 
	Taking $\delta\to 0$ and weak limits of the subsequence of $\{ u_\delta\}_{\delta}$, we obtain an $\underline{\mathbb{C}}^r$-valued $(n,0)$-form $udz$ satisfying $\dbar(udz)=\alpha$ and 
	$$
	\int_\Delta \frac{|u|^2_h}{|z-x|^{2(n+k)}} \frac{(\ai\ddbar|z|^2)^n}{n!}<+\infty.
	$$
	Set $F=\theta f-u$. We have that $F\in H^0_{(2, h)}(\Delta, \underline{\mathbb{C}}^r)$ and 
	$f_x-F_x=u_x\in \mathscr{E}(h)_x\cap \mathfrak{m}_x^{k+1}\cdot \mathscr{O}(\underline{\mathbb{C}}^r)_x$. 
	
	Then, due to the Artin-Rees lemma, we get a positive integer $l$ such that 
	\begin{align*}
	\mathscr{E}(h)_x\cap \mathfrak{m}_x^{k+1}\cdot \mathscr{O}(\underline{\mathbb{C}}^r)_x &= \mathfrak{m}^{k-l+1}_x (\mathfrak{m}_x^l\cdot \mathscr{O}(\underline{\mathbb{C}}^r) \cap \mathscr{E}(h)_x )
	\end{align*}
	holds for $k\geq l-1$. Hence, it follows that 
	\begin{align*}
	\mathscr{E}(h)_x&=\mathscr{F}_x+\mathfrak{m}^{k-l+1}_x (\mathfrak{m}_x^l\cdot \mathscr{O}(\underline{\mathbb{C}}^r) \cap \mathscr{E}(h)_x )\\
	&\subset \mathscr{F}_x+\mathfrak{m}_x\cdot \mathscr{E}(h)_x\\
	&\subset \mathscr{E}(h)_x
	\end{align*}
	for $k>l-1$. Thanks to Nakayama's lemma, we obtain $\mathscr{F}_x=\mathscr{E}(h)_x$. 
\end{proof}

Applying theorem \ref{thm:l2}, we can prove Theorem \ref{thm:vanishing}. 

\begin{proof}[\indent\sc Proof of Theorem \ref{thm:vanishing}]
Let $\mathscr{L}^q$ be the sheaf of germs of $(n, q)$-forms $u$ with values in $E$ and with square-integrable coefficients, 
such that $|u|^2_{(\omega_X, h)}$ is locally integrable, $\dbar u$ can be defined in the sense of currents with square-integrable coefficients, and $|\dbar u|^2_{(\omega, h)}$ is locally integrable. 
Then $(\mathscr{L}^{\bullet}, \dbar)$ is a resolution of the sheaf $K_X\otimes \mathscr{E}(h)$ for the reason that 
we can solve the $\dbar$-equation locally by applying Theorem \ref{thm:l2} on any small polydisc. 
Hence, we have that $\mathscr{L}^\bullet$ is a resolution by acyclic sheaves. 

The compactness of $X$ yields that locally integrable sections are also integrable on $X$. 
Hence, by using Theorem \ref{thm:l2} globally, we also get that $H^q(\Gamma(X, \mathscr{L}^\bullet))=0$ for $q>0$. 
Consequently, we can conclude that $H^q(X, K_X\otimes \mathscr{E}(h))=0$ for $q>0$. 
\end{proof}

\begin{remark}
We see that the $L^2$-estimate in Theorem \ref{thm:l2} also holds in the situation that the base manifold $X$ is Stein. 
Hence, we can apply Theorem \ref{thm:l2} on any small polydisc in the above proof. 
\end{remark}

As an application of Theorem \ref{thm:vanishing} and \ref{thm:dsstrict}, we obtain the following theorem, which generalizes the Griffiths vanishing theorem. 

\begin{theorem}$($= {\rm Theorem \ref{cor:grif}}$)$
Let $(X, \omega_X)$ be a projective manifold and a K\"ahler metric on $X$. If $h$ is strictly Griffiths $\delta_{\omega_X}$-positive in the sense of Definition \ref{def:strictgrif} on $X$, then
$$
H^q(X, K_X\otimes \mathscr{E}(h\otimes \det h))=0.
$$ 
\end{theorem}

Here we introduce the notion of the Lelong number of a singular Hermitian metric on a holomorphic line bundle. 
Usually, the Lelong of a plurisubharmonic function of $\varphi$ at a point $x\in X$ is defined by 
$$
\liminf_{z\to x} \frac{\varphi (z)}{\log |z-x|}
$$
for some coordinate $(z_1, \cdots , z_n)$ around $x$. We also denote by $\nu(\varphi, x)$ the Lelong number of $\varphi$ at $x\in X$. 
It is known that this number is independent of the choice of local coordinates. 

For a semi-positive singular Hermitian metric $g$ on a holomorphic line bundle $L$, we can also define the Lelong number $\nu (g, x)$ of $g$ at $x$ such that 
$$
\nu (g, x):= \liminf_{z\to x}\frac{-\log g (z) }{\log |z-x|}.
$$
Here we regard $g(z)$ as a local semi-positive function. 
Since $g$ is semi-positive, $-\log g(z)$ is a plurisubharmonic function locally. 
Thus, the above definition is reasonable. 
We repeat that this definition is independent of the choice of local coordinates.

There is a relationship between the Lelong number of $\varphi$ and the integrability of $e^{-\varphi}$. 
We introduce the following important result obtained by Skoda in \cite{Sko72}.

\begin{lemma}$($\cite{Sko72}$)$\label{lem:sukoda}
Let $\varphi$ be a plurisubharmonic function. If $\nu(\varphi, x)<1$, 
$e^{-2\varphi}$ is integrable around $x$. 
\end{lemma}

We consider the strictly Nakano $\delta_{\omega_X}$-positive or strictly Griffiths $\delta_{\omega_X}$-positive singular Hermitian metric $h$ again. 
We recall that $\det h$ is a semi-positive singular Hermitian metric on $\det E$ (cf. \cite[Proposition 1.3]{Rau15}).
If the Lelong number of $\det h$ satisfies some good inequalities, 
we have that $\mathscr{E}(h)=\mathscr{O}(E)$ or $\mathscr{E}(h\otimes \det h)=\mathscr{O}(E\otimes \det E)$. 
These properties imply the following vanishing theorems. 

\begin{theorem}\label{thmm:van}
Let $(X, \omega_X)$ be a projective manifold and a K\"ahler metric on $X$. We also let $h$ be a globally strictly Nakano $\delta_{\omega_X}$-positive singular Hermitian metric on $E$.
If $\nu(\det h, x)<2$ for any point $x\in X$, we have $\mathscr{E}(h)=\mathscr{O}(E)$ and 
$$
H^q(X, K_X\otimes E)=0
$$
for $q>0$. 
\end{theorem}

\begin{proof}
By the definition of the Lelong number of a singular Hermitian metric on a holomorphic line bundle, 
we have $\nu(\frac{1}{2}\log \det h^\star, x)<1$ for every $x\in X$. 
From Lemma \ref{lem:sukoda}, 
$$
e^{-\log \det h^\star}= \frac{1}{\det h^\star}
$$
is locally integrable. Locally, we see that 
$$
h= \frac{1}{\det h^\star}\widehat{h}^\star, 
$$
where $\widehat{h}^\star$ is the adjugate matrix of $h^\star$. 
Since $h^\star$ is Griffiths semi-negative, each element of $\widehat{h}^\star$ is locally bounded \cite[Lemma 2.2.4]{PT18}. 
Then it follows that $|u|^2_h$ is locally integrable for any local holomorphic section $u\in \mathscr{O}(E)$ of $E$. 
Therefore, we can conclude that $\mathscr{E}(h)=\mathscr{O}(E)$ and $H^q(X, K_X\otimes E)=0$ for $q>0$ from Theorem \ref{thm:vanishing}. 
\end{proof}
Repeating the above argument and using Theorem \ref{cor:grif}, we can also prove the following theorem. 

\begin{theorem}$($\cite[Corollary 1.4]{Ina18}$)$\label{thmm:van2}
Let $(X, \omega_X)$ be a projective manifold and a K\"ahler metric on $X$. We also let $h$ be a strictly Griffiths $\delta_{\omega_X}$-positive singular Hermitian metric on $E$.
If $\nu(\det h, x)<1$ for any point $x\in X$, we have $\mathscr{E}(h\otimes \det h)=\mathscr{O}(E\otimes \det E)$ and 
$$
H^q(X, K_X\otimes E\otimes \det E)=0
$$
for $q>0$. 
\end{theorem}

As an application of Theorem \ref{thmm:van} and \ref{thmm:van2}, 
we can say that certain vector bundles cannot admit Nakano or Griffiths $\delta_{\omega_X}$-positive singular Hermitian metrics. 
We show the example. 

\begin{example}\label{examp:rei}
Let $(\mathbb{P}^n, \omega_{FS})$ be the $n$-dimensional projective space and the Fubini-Study metric on $\mathbb{P}^n$, where $n\geq 2$. 
Let $Q$ be the vector bundle of rank $n$ over $\mathbb{P}^n$ defined by 
$$
0\to \mathscr{O}(-1)\to \underline{\mathbb{C}}^{n+1}\to Q\to 0,
$$
where $\underline{\mathbb{C}}^{n+1}$ is the trivial vector bundle of rank $n+1$ and $\mathscr{O}(-1)$
is the tautological line bundle. 
There exist the isomorphisms
\begin{align*}
\det Q &\cong \mathscr{O}(1)\\
T\mathbb{P}^n &\cong Q\otimes \det Q.
\end{align*}
Then we can conclude that $Q$ does not admit any Griffiths $\delta_{\omega_{FS}}$-positive singular Hermitian metrics such that their Lelong number is less than $1$ at every point (cf. \cite[Example 5.2]{Ina18}) and $T\mathbb{P}^n$
does not admit any globally Nakano $\delta_{\omega_{FS}}$-positive singular Hermitian metrics such that their Lelong number is less than $2$ at every point for any $\delta>0$. 
Indeed, we have that 
$$
H^q(\mathbb{P}^n, K_{\mathbb{P}^n}\otimes T\mathbb{P}^n)\cong \mathbb{C}\neq 0
$$
if $q=n-1$ (cf. \cite[Chapter VII, Example 8.4]{DemCom}).
\end{example}

\section{Properties of Nakano semi-positivity}\label{sec:property}
In this short section, we discuss the validity of the definition of Nakano semi-positive singular Hermitian metrics. 
We show the following results.

\begin{proposition}\label{prop:linebdl}
Let $L\to X$ be a holomorphic line bundle on a complex manifold $X$. 
We also let $h$ be a $($Griffiths$)$ semi-positive singular Hermitian metric on $L$. 
Then $h$ is globally Nakano semi-positive in the sense of singular Hermitian metrics.
\end{proposition}

\begin{proposition}\label{prop:riemannmen}
Let $S$ be a Riemann surface and $E\to S$ be a holomorphic vector bundle on $S$. 
We also let $h$ be a Griffiths semi-positive singular Hermitian metric on $E$. 
Then $h$ is globally Nakano semi-positive in the sense of singular Hermitian metrics.
\end{proposition}

If $h$ is smooth, Griffiths semi-positivity is equivalent to Nakano semi-positivity in the settings of Proposition \ref{prop:linebdl} and \ref{prop:riemannmen}. 
These propositions imply that our definition of Nakano semi-positivity of singular Hermitian metrics is appropriate when we compare it with already-known positivity notions. 
Repeating the argument in the proof of Theorem \ref{thm:demaillyskoda}, we can prove the above propositions. 
Here we use the same notation as in the proof of Theorem \ref{thm:demaillyskoda}. 

\begin{proof}[\indent\sc Proof of Proposition \ref{prop:linebdl}]
Let $(\Omega, \iota)$ be a Stein coordinate of $X$ such that $L|_{\iota(\Omega)}$ is trivial on $\iota(\Omega)$. We simply write $(\iota^\star L, \iota^\star h)=(L, h)$ on $\Omega$. 
We take an arbitrary K\"ahler metric $\omega_\Omega$, 
an arbitrary smooth plurisubharmonic function $\psi$, and a global holomorphic frame $s$ of $L$ on $\Omega$. 
We define the plurisubharmonic function $\varphi$ on $\Omega$ by 
$$
|s|_{h}=e^{-\varphi}. 
$$
By using a usual regularization technique of convolution or Proposition \ref{prop:bp} and repeating the argument in the proof of Theorem \ref{thm:demaillyskoda}, we get a sequence of smooth plurisubharmonic functions $\{ \varphi_\nu \}_{\nu =1}^\infty$ such that this sequence is decreasing to $\varphi$ on any relatively compact subset of $\Omega$. 
Then, taking an exhaustion of $\Omega$, we can obtain the following estimate 
$$
\int_\Omega |u|^2_{\omega_\Omega}e^{-(\varphi+\psi)}dV_{\omega_\Omega} \leq \int_\Omega \langle B^{-1}_{\omega_\Omega, \psi}f, f \rangle_{\omega_\Omega} e^{-(\varphi +\psi)}dV_{\omega_\Omega}
$$
for any $\dbar$-closed $f\in L^2_{(n,q)}(\Omega, L ; \omega_\Omega, he^{-\psi})$ with the solution $u\in L^2_{(n, q-1)}(\Omega, L ; \omega_\Omega, he^{-\psi})$ of $\dbar u=f$.
Consequently, we complete the proof. 
\end{proof}

\begin{proof}[\indent\sc Proof of Proposition \ref{prop:riemannmen}]
We obtain a sequence of smooth Hermitian metrics, with Griffiths positive curvature, increasing to $h$ on any relatively compact subset again. 
Since $S$ is a Riemann surface, $h_\nu$ is also Nakano semi-positive. 
Hence, repeating the argument in the proof of Theorem \ref{thm:demaillyskoda}, we get 
$$
\int_\Omega |u|^2_{(\omega_\Omega, h)}e^{-\psi}dV_{\omega_\Omega} \leq \int_\Omega \langle B^{-1}_{\omega_\Omega, \psi}f, f \rangle_{(\omega_\Omega, h)} e^{-\psi}dV_{\omega_\Omega}
$$
for any $\dbar$-closed $f\in L^2_{(1,1)}(\Omega, E ; \omega_\Omega, he^{-\psi})$ with the solution $u\in L^2_{(1, 0)}(\Omega, E ; \omega_\Omega, he^{-\psi})$ of $\dbar u=f$. 
\end{proof}

\section{Applications}\label{sec:applications}
In this section, as applications of our definitions and main theorems, we show several results. 
First, we prove that Nakano semi-positivity is preserved with respect to a increasing sequence.
This phenomenon is firstly mentioned in \cite{Ina20}. 
Here we explicitly state the detailed proof. 

\begin{proposition}\label{prop:increasing}
	We let $h$ be a singular Hermitian metric on $E\to X$. 
	Assume that there exists a sequence of smooth Nakano semi-positive metrics $\{ h_\nu\}_{\nu=1}^\infty$ increasing to $h$ pointwise. 
	Then $h$ is globally Nakano semi-positive in the sense of Definition \ref{def:nakanosemipositive}. 
\end{proposition}

\begin{proof}
	It is well known that Griffiths semi-positivity satisfies this property. 
	Hence, we know that $h$ is Griffiths semi-positive, and it is enough to show that $h$ satisfies the condition in Defintion \ref{def:nakanosemipositive}. 
	
	Fix a Stein coordinate $(\Omega, \iota)$ which trivializes $E|_\Omega\cong \Omega \times \mathbb{C}^r$, 
	a K\"ahler form $\omega_\Omega$ on $\Omega$,
	a smooth strictly plurisubharmonic function $\psi$ on $\Omega$ and 
	a $\dbar$-closed $f\in L^2_{(n,q)}(\Omega, E; \omega_{\Omega}, he^{-\psi})$.
	Here we omit $\iota$ for simplicity. 
	
	Since $h_\nu$ is Nakano semi-positive, 
	we get a solution $u_\nu$ of $\dbar u_\nu=f$ satisfying 
	\begin{align*}
	\int_\Omega |u_\nu|^2_{(\omega_\Omega, h_\nu)}e^{-\psi}dV_{\omega_\Omega} &\leq \int_\Omega \langle B_{\omega_\Omega, \psi}^{-1}f, f\rangle_{(\omega_\Omega, h_\nu)}e^{-\psi}dV_{\omega_\Omega}\\
	&\leq \int_\Omega \langle B_{\omega_\Omega, \psi}^{-1}f, f\rangle_{(\omega_\Omega, h)}e^{-\psi}dV_{\omega_\Omega}<+\infty
	\end{align*}
	for each $\nu\in \mathbb{N}$. 
	Note that the right-hand side of the inequality above has an upper bound independent of $\nu$. 
	We also remark that $\{ u_\nu\}_{\nu\geq j}$ forms a bounded sequence in $L^2_{(n,q)}(\Omega, E; \omega_{\Omega}, h_je^{-\psi})$ due to the monotonicity of $\{ h_\nu\}$. 
	Hence, we can get a weakly convergent subsequence $\{ u_{\nu_k}\}_{k=1}^\infty$ by using a diagonal argument and the monotonicity of $\{ h_\nu\}$.
	We have that $\{ u_{\nu_k}\}_{k=1}^\infty$ weakly converges in $L^2_{(n,q)}(\Omega, E; \omega_{\Omega}, h_\nu e^{-\psi})$ for every $\nu$. 
	Hence, the weak limit denoted by $u_\infty$ satisfies $\dbar u_\infty= f$ and 
	$$
	\int_\Omega |u_\infty|^2_{(\omega_\Omega, h)}e^{-\psi}dV_{\omega_\Omega} \leq \int_\Omega \langle B_{\omega_\Omega, \psi}^{-1}f, f\rangle_{(\omega_\Omega, h)}e^{-\psi}dV_{\omega_\Omega}
	$$
	due to the monotone convergence theorem, which completes the proof. 
\end{proof}

This proposition also holds for globally strictly Nakano $\delta_{\omega_X}$-positive metrics when $X$ has a K\"ahler metric $\omega_X$. 
As an application of Proposition \ref{prop:increasing}, we prove the Nakano semi-positivity of some sort of direct image bundle, which corresponds to a singular version of Berndtsson's result \cite{Ber09}.  

\begin{theorem}$($= Theorem \ref{mainthm:singber}$)$\label{thm:singber}
	Let $U\subset \mathbb{C}^n_{\{ t\}}$ and $\Omega \subset \mathbb{C}^m_{\{ z\}}$ be bounded domains, and $\varphi$ be a locally bounded plurisubharmonic function on $\overline{U\times \Omega}$.
	We also let $\Omega$ be pseudoconvex. 
	Set, for each $t\in U$, $A^2_t:= \{ f\in \mathscr{O}(\Omega) \mid \| f\|_t^2:=\int_\Omega |f|^2e^{-\varphi(t, \cdot)} <+\infty \}$ and $A^2:=\coprod_{t\in U}A^2_t $.  
	Then $(A^2, \| \cdot \|)$ is globally Nakano semi-positive in the sense of Definition \ref{def:nakanosemipositive}. 
\end{theorem}

\begin{proof}
	Note that $\varphi$ is a plurisubharmonic function on some open neighborhood of $\overline{U\times \Omega}$ and bounded on $U\times \Omega$.
	Take an approximating sequence of smooth plurisubharmonic functions $\{ \varphi_\nu\}_{\nu=1}^\infty$ decreasing to $\varphi$ on $U\times \Omega$. 
	We also let $\| \cdot\|_\nu$ denote a Hermitian metric on $A^2$ associated with $\varphi_\nu$.
	Then we have $(A^2, \| \cdot \|_\nu)$ is Nakano semi-positive thanks to Berndtsson's theorem \cite{Ber09}. 
	Since $\| \cdot \|_\nu$ is increasing to $\| \cdot \|$, due to Proposition \ref{prop:increasing}, we can conclude that $(A^2, \| \cdot \|)$ is globally Nakano semi-positive.
\end{proof}

\begin{remark}
	The local boundedness of $\varphi$ is just a technical assumption which ensures that $A^2$ is a trivial bundle of infinite rank. 
\end{remark}

\begin{remark}
	In our formulation in this article, we only deal with a finite rank vector bundle, but the vector bundle $A^2$ is of infinite rank. 
	However, we can naturally extend Definition \ref{def:nakanosemipositive} and the characterization in Proposition \ref{prop:nakano}  to the case that $E$ is of infinite rank. 
	Thus the proof above is fine. 
	See \cite[Theorem 1.1, Section 2.3]{DNWZ20} for the detailed explanation. 
\end{remark}


%

As an analogue of Theorem \ref{thm:singber} in the global setting, we propose the following conjecture. 

\begin{conjecture}
	Let $f: X\to Y$ be a projective surjective morphism between complex manifolds. 
	Suppose that there exists a holomorphic line bundle with a singular Hermitian metric of semi-positive curvature $(L, h)$ over $X$.
	Then the pushforward sheaf $f_\star(K_{X/Y}\otimes L \otimes \mathscr{I}(h))$ admits a canonical singular Hermitian metric, which is globally Nakano semi-positive in the sense of Defintion \ref{def:nakanosemipositive}. 
\end{conjecture}

It is known that the pushforward sheaf has a canonical ``Griffiths" semi-positive singular Hermitian metric \cite[Theorem 21.1]{HPS18}. 
 
Next, we consider the following situation. 

\begin{proposition}\label{prop:kakucho}
	Let $h$ be a Griffiths semi-positibe singular Hermitian metric on $E\to X$. 
	Suppose that there exists a proper analytic subset $S\subset X$ such that $X\setminus S$ is Stein and $h$ is globally Nakano semi-positive on $X\setminus S$. 
	Then we have that $h$ is globally Nakano semi-positive on $X$ as well. 
\end{proposition}

\begin{proof}
	Take an arbitrary Stein coordinate $\Omega \hookrightarrow X$ which trivializes $E|_\Omega \cong \Omega \times \mathbb{C}^r$, 
	a K\"ahler metric $\omega_\Omega$, 
	a smooth strictly plurisubharmonic function $\psi$ on $\Omega$ and 
	a $\dbar$-closed $f\in L^2_{(n,q)}(\Omega, E; \omega_\Omega, he^{-\psi})$. 
	We only need to consider the case that $\Omega \cap S \neq \emptyset$. 
	Since $\Omega \setminus S$ is also Stein and $E|_{\Omega \setminus S}$ is trivial, 
	we can solve the equation $\dbar u=f$ with the estimate 
	\begin{align*}
	\int_{\Omega \setminus S}|u|^2_{(\omega_\Omega, h)}e^{-\psi}dV_{\omega_\Omega}
	&\leq \int_{\Omega \setminus S} \langle B^{-1}_{\omega_{\Omega}, \psi}f, f\rangle_{(\omega_\Omega, h)}e^{-\psi} dV_{\omega_\Omega}\\
	&\leq \int_{\Omega} \langle B^{-1}_{\omega_{\Omega}, \psi}f, f\rangle_{(\omega_\Omega, h)}e^{-\psi} dV_{\omega_\Omega}<+\infty
	\end{align*}
	on $\Omega \setminus S$. 
	Set $u=0$ on $S$. 
	Repeating the argument in the proof of Theorem \ref{thm:l2},
	we have $u\in L^2_{(n, q-1)}(\Omega, E; \omega_\Omega, he^{-\psi})$, $\dbar u=f$ on $\Omega$ and 
	$$
	\int_{\Omega}|u|^2_{(\omega_\Omega, h)}e^{-\psi}dV_{\omega_\Omega}
	\leq \int_{\Omega} \langle B^{-1}_{\omega_{\Omega}, \psi}f, f\rangle_{(\omega_\Omega, h)}e^{-\psi} dV_{\omega_\Omega}.
	$$ 
	Then we finish the proof. 
\end{proof}

Note that the proposition above holds for globally strictly Nakano $\delta_{\omega_X}$-positive singular Hermitian metrics when $X$ is a K\"ahler manifold. 
Applying Proposition \ref{prop:kakucho}, we can prove Theorem \ref{thm:bigvanishing}. 

\begin{proof}[\indent\sc Proof of Theorem \ref{thm:bigvanishing}]
	Since $E$ is a $V$-big vector bundle, thanks to Definition \ref{def:lbigvbig} and Proposition \ref{prop:baseloci}, we can construct a singular Hermitian metric $\widehat{h}$ on $\mathscr{O}_{\mathbb{P}(E)}(1)$, a positive constant $\varepsilon>0 $ and a proper analytic subset $\widehat{S}\subset \mathbb{P}(E)$ satisfying the following conditions: 
	\begin{enumerate}
		\item $\widehat{h}$ is smooth on $\mathbb{P}(E)\setminus \widehat{S}$, 
		\item $\ai\Theta_{(\mathscr{O}_{\mathbb{P}(1)}, \widehat{h})} \geq \varepsilon \omega_{\mathbb{P}(E)}$,
		\item $\pi(\widehat{S})\neq X$,
	\end{enumerate}
where $\omega_{\mathbb{P}(E)}$ is a fixed K\"ahler metric on $\mathbb{P}(E)$. 
Consider the isomorphism $\pi_\star(K_{\mathbb{P}(E)/X}\otimes \mathscr{O}_{\mathbb{P}}(r+m))\cong S^mE\otimes \det E$, where $K_{\mathbb{P}(E)/X}=K_{\mathbb{P}(E)}\otimes \pi^\star (K_X^{-1})$ is the relative canonical bundle. 
Then $S^mE \otimes \det E$ admits the $L^2$ metric $h_m$ associated with $\widehat{h}^m$. 
We fix an analytic subset $S\subsetneq X$ such that $X\setminus S$ is Stein and $S\supset \pi(\widehat{S})$. 
Due to the construction above, we have that $h_m$ is smooth on $X\setminus S$ and Griffiths semi-positive on $X$ \cite{BP08}, \cite{PT18}. 
Moreover, $h_m$ is smooth Nakano positive on $X\setminus S$ thanks to \cite{Ber09}, and actually has global strict Nakano $\delta_{\omega_X}$-positivity for some $\delta >0$. 
In order to prove the latter property, we take a sufficiently small open subset $U\subset X\setminus S$ such that $\omega_X$ is $\ddbar$-exact on $U$, and take its potential $\varphi_U$, i.e. $\ai \ddbar \varphi_U=\omega_X|_U$. 
Taking a positive constant $\delta>0$ satisfying 
$$
\ai\Theta_{(\mathscr{O}_{\mathbb{P}(1)}, \widehat{h})} \geq \varepsilon \omega_{\mathbb{P}(E)}\geq \frac{\delta}{m}\pi^\star \omega_X,
$$
we see that $\widehat{h}e^{\frac{\delta}{m}(\pi^\star \varphi_U)}$ is semi-positive on $\pi^{-1}(U)$, which means that $h_me^{\delta \varphi_U}$ is Nakano semi-positive on $U$. 
Then, applying Proposition \ref{prop:kakucho}, we can conclude that $h_m$ is a globally strictly Nakano $\delta_{\omega_X}$-positive ``singular" Hermitian metric over $X$. 
The vanishing theorem is a direct corollary of Theorem \ref{thm:vanishing}.
\end{proof}

\section{Related problems}\label{sec:problems}
In the last section, we propose important problems related to the main theorems. 

First of all, we consider Proposition \ref{prop:bp}. 
This regularization technique is a fundamental tool to study Griffiths semi-positive singular Hermitian metrics. 
However, the way to regularize a Nakano semi-positive singular Hermitian metric is not known. 
Then we propose the following problem. 

\begin{question}\label{ques:kinji}
Let $E$ be a trivial vector bundle over a polydisc $\Delta \subset \mathbb{C}^n$. 
We also let $h$ be a Nakano semi-positive singular Hermitian metric on $E$. 
Then, can we construct a sequence of smooth Hermitian metrics, with Nakano positive curvature, increasing to $h$ on any smaller polydiscs?
\end{question}

Next, we think the Demailly-Nadel type vanishing theorem. 
In general, this vanishing theorem is established on weakly pseudoconvex manifolds. 
Then we can expect that the main theorems also hold on weakly pseudoconvex manifolds. 

\begin{question}
Let $(E, h)$ be a holomorphic vector bundle and a strictly Nakano positive singular Hermitian metric over a weakly pseudoconvex manifold $X$. 
Then can we obtain $L^2$-estimates and vanishing theorems with coefficients in $E$ on $X$?
\end{question}

Next, we consider the definition of Nakano semi-positivity. 
In this article, we assume the Griffiths semi-positivity of Nakano semi-positive singular Hermitian metrics. 
In the smooth setting, it is clear that a Nakano semi-positive Hermitian metric is always Griffiths semi-positive. 
However, in the singular setting, we do not know whether Nakano semi-positivity yields Griffiths semi-positivity. 

\begin{question}\label{ques:grif}
We let $h$ satisfy the condition in Definition \ref{def:nakanosemipositive} without assuming the Griffiths semi-positivity of $h$. 
Can we say that $h$ is Griffiths semi-positive?
\end{question}

We remark that there exists a result related to Question \ref{ques:grif} (cf. \cite[Theorem 3.5]{HI19} and \cite[Theorem 1.2]{DNWZ20}). 

Next, we consider the conditions $\{$(2-k)$\}_{1\leq k \leq n}$ in Remark \ref{rem:2k}. 
As already mentioned, these conditions are equivalent to each other when $h$ is a smooth Hermitian metric.
We expect that this equivalence is also valid even when $h$ is a singular Hermitian metric. 

\begin{question}\label{ques:2k}
Prove the equivalence of the conditions $\{$(2-k)$\}_{1\leq k \leq n}$ in the case that $h$ is a singular Hermitian metric.
\end{question}

At last, we consider the equivalence of the global Nakano semi-positivity and the 
local Nakano semi-positivity. 

\begin{question}\label{ques:globalocal}
Show the equivalence of the global Nakano semi-positivity in Definition \ref{def:nakanosemipositive}
and the local Nakano semi-positivity in Definition \ref{def:localnakano} for singular Hermitian metrics. 
\end{question}

If we can verify Question \ref{ques:kinji}, we can also prove Question \ref{ques:grif} and \ref{ques:2k} by using the regularization technique. 
In fact, Question \ref{ques:grif} and \ref{ques:2k} are correct if $h$ is smooth. 
Then, if we can take a sequence of smooth Hermitian metrics with Nakano positive curvature, 
we verify these questions by repeating the argument in the proof of Theorem \ref{thm:demaillyskoda}. 
Therefore, Question \ref{ques:kinji} is a crucial problem.




\begin{thebibliography}{99}
 \bibitem[AV65]{AV65} A.~Andreotti and E.~Vesentini, \emph{Carleman estimates for the Laplace-Beltrami equation in complex manifolds}, Publ. Math. I.H.E.S. \textbf{25} (1965), 81-130.
 \bibitem[BKKMSU15]{BKKMSU15} T.~Bauer, S.J.~Kov\'acs, A.~K\"uronya, E.C.~Mistretta, T.~Szemberg and S.~Urbinati, \emph{On positivity and base loci of vector bundles}, European Journal of Mathematics \textbf{1} (2015), 229-249.
 \bibitem[Ber98]{Ber98} B.~Berndtsson, \emph{Prekopa's theorem and Kiselman's minimum principle for plurisubharmonic functions}, Math. Ann. \textbf{312} (1998), 785-792. 
 \bibitem[Ber09]{Ber09} B.~Berndtsson, \emph{Curvature of vector bundles associated to holomorphic fibrations}, Ann. of Math. \textbf{169} (2009), no.~2, 531-560. 
 \bibitem[Ber10]{Ber10}
B.~Berndtsson, \emph{An introduction to things $\dbar$, Analytic and algebraic geometry},
IAS/Park City Math. Ser., vol. 17, Amer. Math. Soc., Providence, RI, 2010, pp. 7-76.
 \bibitem[BP08]{BP08} B.~Berndtsson and M.~P{\u{a}}un, \emph{Bergman kernels and the pseudoeffectivity of relative canonical bundles}, Duke Math. J. \textbf{145} (2008), no.~2, 341--378.
 \bibitem[deC98]{deC98} M.~A.~A. de~Cataldo, \emph{Singular {H}ermitian metrics on vector bundles}, J. Reine Angew. Math. \textbf{502} (1998), 93--122.
 \bibitem[Dem82]{Dem82} J.-P. Demailly, \emph{Estimations $L^2$ pour l'op\'erateur $\dbar$ d'un fibr\'e vectoriel holomorphe semi-positif au dessus d'une vari\'et\'e K\"ahl\'erienne compl\`ete}, Ann. Sci. Ec. Norm. Sup. \textbf{15} (1982), 457-511.
 \bibitem[Dem93]{Dem93} J.-P. Demailly, \emph{A numerical criterion for very ample line bundles}, J. Differential Geom. \textbf{37} (1993), 323-374. 
 \bibitem[Dem12]{Dem12}
J.-P. Demailly, \emph{Analytic methods in algebraic geometry}, Surveys of Modern
  Mathematics, vol.~1, International Press, Somerville, MA; Higher Education
  Press, Beijing, 2012.
 \bibitem[Dem-book]{DemCom} J.-P. Demailly, \emph{Complex analytic and differential geometry}, http://www-fourier.ujf-grenoble.fr/~demailly/manuscripts/agbook.pdf.
 \bibitem[DS]{DS}
J.-P. Demailly and H.~Skoda, \emph{Relations entre les notions de positivit\'e de P.~A.~Griffiths et de S.~Nakano},
S\'eminaire P.~Lelong-H.~Skoda (Analyse), ann\'ee 1978/79, Lecture notes in Math. no \textbf{822}, Springer-Verlag, Berlin (1980) 304-309.
 \bibitem[DNW21]{DNW19} F.~Deng, J.~Ning, and Z.~Wang, \emph{Characterizations of plurisubharmonic functions}, Sci. China Math. (2021).
 https://doi.org/10.1007/s11425-021-1873-y.
 \bibitem[DNWZ20]{DNWZ20} F.~Deng, J.~Ning, Z.~Wang, and X.~Zhou, \emph{Positivity of holomorphic vector bundles in terms of $L^p$-conditions of $\bar{\partial}$}, arXiv:2001.01762. 
 \bibitem[DWZZ18]{DWZZ18} F.~Deng, Z.~Wang, L.~Zhang, and X.~Zhou, \emph{New characterizations of plurisubharmonic functions and positivity of direct image sheaves}, arXiv:1809.10371.
 \bibitem[DWZZ19]{DWZZ19} F.~Deng, Z.~Wang, L.~Zhang, and X.~Zhou, \emph{A new proof of Kiselman's minimum principle for plurisubharmonic functions}, Comptes Rendus Math\'ematique \textbf{357} (2019), no.~4, 345-348. 
 \bibitem[HPS18]{HPS18} C.~Hacon, M.~Popa, and C.~Schnell, \emph{Algebraic fiber spaces over Abelian varieties:
 around a recent theorem by Cao and Paun}, Contemp. Math. \textbf{712} (2018), 143-195. 
 \bibitem[Hor65]{Hor65} L.~H\"ormander, \emph{$L^2$ estimates and existence theorems for the $\dbar$ operator}, Acta Math. \textbf{113} (1965), 89-152. 
 \bibitem[HI20]{HI19} G.~Hosono and T.~Inayama, \emph{A converse of H\"ormander's $L^2$-estimate and new positivity notions for vector bundles}, Sci. China Math. (2020). 
 https://doi.org/10.1007/s11425-019-1654-9.
 \bibitem[Ina20]{Ina18} T.~Inayama, \emph{$L^2$ estimates and vanishing theorems for holomorphic vector bundles equipped with singular Hermitian metrics}, Michigan Math. J. \textbf{69} (2020), 79-96.
 \bibitem[Ina21]{Ina20} T.~Inayama, \emph{From H\"ormander's $L^2$-estimates to partial positivity}, Comptes Rendus. Math\'ematique \textbf{359} (2021), no.~2, 169-179. 
 \bibitem[Iwa18]{Iwa18} M.~Iwai, \emph{Nadel-Nakano vanishing theorems of vector bundles with singular Hermitian metrics}, arXiv:1802.01794, to appear in Annales de la Facult\'e des sciences de Toulouse: Math\'ematiques.
 \bibitem[Nad90]{Nad90} A.~M.~Nadel, \emph{Multiplier ideal sheaves and K\"ahler-Einstein metrics of positive scalar curvature}, Ann. of Math. (2) \textbf{132} (1990), no.~3, 549-596.
 \bibitem[PT18]{PT18} M.~P{\u{a}}un and S.~Takayama, \emph{Positivity of twisted relative pluricanonical bundles and their direct images}, J. Algebraic Geom. \textbf{27} (2018), 211-272. 
 \bibitem[Rau15]{Rau15} H.~{Raufi}, \emph{Singular hermitian metrics on holomorphic vector bundles}, Ark. Mat. \textbf{53} (2015), no.~2, 359-382.
 \bibitem[Ser56]{Ser56} J.-P. Serre, \emph{G\'eom\'etrie alg\'ebrique et g\'eom\'etrie analytique}, Annales de l'Institut Fourier \textbf{6} (1956), 1-42. 
\bibitem[Siu76]{Siu76} Y.~T.~Siu, \emph{Every Stein subvariety admits a Stein neighborhood}, Invent. Math. \textbf{38} (1976/77), 89-100.
\bibitem[Sko72]{Sko72}H.~Skoda, \emph{Sous-ensembles analytiques d'ordre fini ou infini dans $\mathbb{C}^n$}, Bull. Soc. Math. France, \textbf{100} (1972), 353-408.
\end{thebibliography}
\end{document}